\documentclass[hidelinks,onefignum,onetabnum]{siamart220329}

\usepackage{verbatim}
\usepackage{url}
\usepackage{subfigure}
\usepackage{epstopdf}
\usepackage{hyperref}
\usepackage{verbatim}
\usepackage{latexsym, graphicx, epsfig}
\usepackage{pgfplots}
\usepackage{url}
\usepackage{mathrsfs}
\usepackage{color}
\usepackage{array}
\usepackage{amsfonts,amsmath,amssymb}
\usepackage{float}
\usepackage{multirow}
\usepackage{booktabs}
\usepackage{subcaption}
\usepackage{tikz}
\usepackage{diagbox}
\usepackage[toc,page]{appendix}
\usepackage{algorithm}
\usepackage{algpseudocode}
\usepackage[numbers,sort&compress]{natbib}

\numberwithin{equation}{section}
\numberwithin{figure}{section}
\numberwithin{table}{section}

\newcommand{\sanshu}{ |\!|\!| }

\newtheorem{exmp}{Example}[section]
\newtheorem{remark}{Remark}[section]
\allowdisplaybreaks[1]

\headers{APDG scheme for semiconductor Boltzmann equation}{H. Ding and L. Liu and X. Zhong}
\title{Asymptotic-preserving and positivity-preserving discontinuous Galerkin method for the semiconductor Boltzmann equation in the diffusive scaling
\thanks{Submitted to the editors February, 2025. 
\funding{L. Liu acknowledges the support by National Key R\&D Program of China (2021YFA1001200), Ministry of Science and Technology in China, Early Career Scheme (24301021) and General Research Fund (14303022 \& 14301423) funded by Research Grants Council of Hong Kong. X. Zhong acknowledges the support by the NSFC Grant (12272347). 
}}}

\author{
Huan Ding \thanks{School of Mathematical Sciences, Zhejiang University, Zhejiang, China(\email{dinghuan@zju.edu.cn}).} 
\and  
Liu Liu \thanks{Department of Mathematics, The Chinese University of Hong Kong, Hong Kong, China
  (\email{liuliu@cuhk.edu.hk}). }
  \and
 Xinghui Zhong \thanks{School of Mathematical Sciences, Zhejiang University, Zhejiang, China (\email{zhongxh@zju.edu.cn}). }
}

\ifpdf
\hypersetup{
  pdftitle={Efficient scheme for semiconductor Boltzmann equation},
}
\fi

\begin{document}

 \maketitle

\begin{abstract}
In this paper, we develop an asymptotic-preserving and positivity-preserving discontinuous Galerkin (DG) method for solving the semiconductor Boltzmann equation in the diffusive scaling. We first formulate the diffusive relaxation system based on the even-odd decomposition method, which allows us to split into one relaxation step and one transport step. We adopt a robust implicit scheme that can be explicitly implemented for the relaxation step that involves the stiffness of the collision term, while the third-order strong-stability-preserving Runge-Kutta method is employed for the transport step. We couple this temporal scheme with the DG method for spatial discretization, which provides additional advantages including high-order accuracy, $h$-$p$ adaptivity, and the ability to handle arbitrary unstructured meshes. A positivity-preserving limiter is further applied to preserve physical properties of numerical solutions. The stability analysis using the even-odd decomposition is conducted for the first time. We demonstrate the accuracy and performance of our proposed scheme through several numerical examples.

\end{abstract}



\section{Introduction}
Kinetic equations have been widely used in various areas, including rarefied gas, plasma physics, astrophysics, semiconductor device modeling, social and biological sciences \cite{Semi-Book}. These equations describe the non-equilibrium dynamics of a system composed of a large number of particles and bridge atomistic and continuum models in the hierarchy of multiscale modeling. The Boltzmann-type equation, as one of the most representative models in kinetic theory, provides a power tool to describe molecular gas dynamics, radiative transfer, plasma physics, and polymer flow \cite{A15}. They have significant impacts in designing, optimization, control, and inverse problems. For example, it can be used in the design of semiconductor devices, topology optimization of gas flow channel, or risk management in quantitative finance \cite{CS21}. The linearized Boltzmann equation, particularly the semiconductor Boltzmann equation, which is the focus of this study, plays a pivotal role in semiconductor device modeling \cite{Ansgar}.

From the computational point of view, the major challenge in simulating kinetic equations arises from the presence of multiple temporal and spatial scales. This variation in scales is typically characterized by the Knudsen number $\varepsilon$, a dimensionless quantity representing the mean free path. The Knudsen number can vary by many orders of magnitude within the computational domain, and drives the system towards drift-diffusion equations as it goes to zero. In the past two decades, asymptotic-preserving (AP) schemes have been a popular and robust computational approach for kinetic and hyperbolic problems. These schemes preserve the asymptotic transition from one scale to another at the discrete level, ensuring accuracy across regimes. For a comprehensive review, see \cite{AP-Review}. The main advantage of AP schemes lies in their efficiency across all regimes, especially in the hydrodynamic or diffusive regimes. Unlike traditional methods, AP schemes do not meed the numerical resolution of small physical parameters yet can still accurately capture the macroscopic behavior. AP schemes were first developed for   steady neutron transport problems in \cite{Jin1991,Jin99}, later applied to non-stationary transport problems in \cite{JPT1998, JPT2000, Klar},  and have since been successfully applied to a wide range of other problems.


In this paper, we aim to construct an efficient AP scheme within the even-odd decomposition framework using the discontinuous Galerkin (DG) spatial discretization. DG methods are a class of finite element methods using discontinuous basis functions. They were first proposed in \cite{reed_triangular_1973} to solve linear transport equations, and later were extended to hyperbolic conservation laws in \cite{cockburn_runge-kutta_1991, cockburn_tvb_1989-1, cockburn_tvb_1989,cockburn_runge-kutta_1990,cockburn_rungekutta_1998}, where a framework using nonlinearly stable high order Runge-Kutta (RK) time discretizations \cite{shu_efficient_1988} and DG discretization in space to easily solve nonlinear time dependent problems, resulting the so-called RKDG method. DG methods offer several advantages including  high-order accuracy, $h$-$p$ adaptivity, and the ability to handle arbitrary unstructured mesh.
DG methods have gained  popularity in a wide range of mathematical models. For kinetic problems, we refer to \cite{cheng_energy-conserving_2014,cheng_energy-conserving_2014-1,cheng_energy-conserving_2015, Xiong1, Xiong2, Peng,cheng_numerical_2015, DG_IMEX, LiF_IMEXDG,filbet_conservative_2022,yin_highly_2023, Ye_HPDG, DG_VA_DFP} and the references therein for an incomplete list.

The goal of this paper is to design high-order asymptotic-preserving DG (APDG) methods for the semiconductor Boltzmann equation. Inspired by the work in \cite{jin_discretization_2000}, we first derive the diffusive relaxation system based on the even-odd decomposition method, which can be split into one relaxation step involving the stiffness of the collision term and one transport step. Then we adopt a robust explicitly implementable implicit scheme \cite{jin_discretization_2000} for the relaxation step and the third-order strong-stability-preserving  RK (SSPRK) method \cite{shu_total-variation-diminishing_1988} for the transport step. The main novelty is that we couple this efficient temporal discretization method with high-order DG method for spatial discretization, together with  the Hermite quadrature rule for the velocity discretization, rendering a high-order APDG scheme for the model problem. We further apply a positivity-preserving limiter to preserve physical properties of the solution. Moreover, we prove that the proposed fully discrete scheme is stable under a suitable CFL condition.  
It is worth emphasizing that within the even-odd decomposition framework, to the authors' best knowledge, this is the {\it first time} the stability analysis is conducted.
Numerical examples are provided to demonstrate the performance of the proposed method.


We summarize the features of the proposed APDG scheme in the following:
\begin{itemize}
\item The scheme is high-order accurate on the fully discrete level, and  uniformly accurate for different $\varepsilon$; 
\item The scheme is positivity-preserving, i.e., $f^n \geq 0$ holds for any time level $n$; 
\item The scheme is AP;
\item The scheme is stable under a suitable CFL condition. 
\end{itemize}

The rest of the paper is organized as follows: In Section
\ref{sec:1}, we introduce the linear semiconductor Boltzmann equations and  the AP temporal discretization method in the even-odd decomposition framework. In Section \ref{sec:full}, we present the fully discrete scheme with the DG method for the spatial discretization, combining with a positivity-preserving limiter.  Section \ref{sec:3} is dedicated to the stability and AP property analysis of the proposed APDG scheme. Several numerical experiments are provided in Section \ref{sec:4} to demonstrate efficiency and accuracy of this new approach. Concluding remarks are given
in Section \ref{sec:conclusion}.

\section{Model equations and temporal discretization}
\label{sec:1}
In this section, we introduce the model equations and  the AP temporal discretization scheme.

The Boltzmann equation is one of the most fundamental kinetic equations which describes the time evolution of the particle probability density distribution.
Under proper scaling, the dimensionless Boltzmann equation reads as
\begin{equation}\label{eq:Boltzmann}
    \varepsilon\partial_t f+v\cdot\nabla_xf-E\cdot\nabla_vf=\frac1\varepsilon\mathcal{Q}(f),
\end{equation}
where $f(x,v,t)$ is the probability density function of particles at position $x\in\mathbb{R}^d$ with velocity $v\in\mathbb{R}^d$ at time $t>0$, with $d$ representing the dimension of the computational field. The 
constant $\varepsilon$ is the Knudsen number, and $E$ is the electric field. 
The Boltzmann collision operator $\mathcal{Q}(f)$ is  given by
\begin{equation}\label{eq:Boltz_operator}
    \mathcal{Q}(f) =\int_\mathbb{R^d} \sigma(v,w) (M(v)f(w)-M(w)f(v))\, dw,
\end{equation}
where $\sigma$ is the anisotropic scattering kernel satisfying $\sigma(v,w)=\sigma(w,v)>0$, and $M$ is the normalized Maxwellian distribution defined as
\begin{equation}\label{eq:Maxwellian}
    M(v) = \frac{1}{\sqrt{2\pi}}e^{-\frac{v^2}{2}}.
\end{equation}
We further define the collision frequency  $\lambda$ as
\begin{equation}
    \lambda(x,v) = \int_\mathbb{R^d} \sigma(v,w)M(w)\,dw,
\end{equation}
and assume that it satisfies $\lambda \leq \mu$ for some positive constant $\mu$. More details about this model can be found in \cite{Semi-Book}.

In the following subsections, we briefly review the process to formulate the diffusive relaxation system based on the even-odd decomposition method \cite{JPT1998, JPT2000, jin_discretization_2000} and introduce the AP temporal discretization scheme, as a preparation for deriving the AP fully discrete scheme in Section \ref{sec:full}. For simplicity of discussion, we consider one-dimensional velocity variable $v\in\mathbb{R}$ as an illustrative example to show the main idea. We refer readers to 
 \cite{JPT1998, JPT2000, jin_discretization_2000} for more  details.

\subsection{Even-odd decomposition and diffusive relaxation system}
\label{sec:evenodd}
We first split \eqref{eq:Boltzmann} into two equations for $v$ and $-v$, respectively,
\begin{subequations}\label{eq:pmv}
    \begin{align}
        \varepsilon\partial_t f(v)+v\cdot\nabla_xf(v)-E\cdot\nabla_vf(v)&=\frac1\varepsilon\mathcal{Q}(f)(v),\\
        \varepsilon\partial_t f(-v)-v\cdot\nabla_xf(-v)+E\cdot\nabla_vf(-v)&=\frac1\varepsilon\mathcal{Q}(f)(-v).
    \end{align}
\end{subequations}
Without abuse of notations, here and after, we adopt simpler notations such as $f(v)$ to represent $f(x,v,t)$. 


Then we define the even parity $r$ and odd parity $j$ as
\begin{subequations}\label{eq:def_even-odd}
    \begin{align}
        r(t,x,v)&=\frac12\left(f(t,x,v)+f(t,x,-v)\right),\\
        j(t,x,v)&=\frac1{2\varepsilon}(f(t,x,v)-f(t,x,-v)),
    \end{align}
\end{subequations}
which, by adding and subtracting the two equations in \eqref{eq:pmv}, leads to 
\begin{subequations}
\label{eq:rj}
    \begin{align}
        \partial_tr+v\cdot\nabla_xj-E\cdot\nabla_vj&=\frac1{\varepsilon^2}\mathcal{Q}(r),\\
        \partial_tj+\frac1{\varepsilon^2}(v\cdot\nabla_xr-E\cdot\nabla_vr)&=-\frac1{\varepsilon^2}\lambda j.
    \end{align}
\end{subequations}
The density  $\rho$ can be expressed in terms of the new variables as
\begin{align}
    \label{eq:density}
    \rho = \int_{-\infty}^{\infty}f(v)dv=\int_{-\infty}^{\infty} r(v)dv.
\end{align}

As in \cite{jin_discretization_2000,JPT2000}, \eqref{eq:rj} can be rewritten into the following diffusive relaxation system 
\begin{equation}\label{eq:diffusive_relaxation_system}
    \begin{split}
        \partial_tr+v\cdot\nabla_xj-E\cdot\nabla_vj & =\frac1{\varepsilon^2}\mathcal{Q}(r),\\
        \partial_tj+\phi(v\cdot\nabla_xr-E\cdot\nabla_vr)
        & =-\frac1{\varepsilon^2}\left(\lambda j+(1-\varepsilon^2\phi)(v\cdot\nabla_xr-E\cdot\nabla_vr)\right),
    \end{split}
\end{equation}
where $\phi$ can be chosen in the simple form of $\phi = \min\left\{1,\frac{1}{\varepsilon^2}\right\}$. 

\subsection{Temporal discretization}
In this subsection, we present the AP temporal discretization scheme for the diffusive relaxation system \eqref{eq:diffusive_relaxation_system}.

As in \cite{JPT1998,JPT2000},  the system \eqref{eq:diffusive_relaxation_system} can employ the conventional splitting resulting  one relaxation step
\begin{subequations}
\label{eq:relaxation_step} 
   \begin{align}
   \partial_tr & =\frac1{\varepsilon^2}\mathcal{Q}(r), \label{eq:relaxation_step_1}\\[2pt]
        \partial_tj &=-\frac1{\varepsilon^2}\left(\lambda j+(1-\varepsilon^2\phi)(v\cdot\nabla_xr-E\cdot\nabla_vr)\right),\label{eq:relaxation_step_2}
 \end{align}
\end{subequations}
and one transport step
\begin{subequations} \label{eq:transport_step}
    \begin{align} 
    \partial_tr+v\cdot\nabla_xj-E\cdot\nabla_vj&=0,\label{eq:transport_step_1} \\[2pt]
    \partial_tj+\phi(v\cdot\nabla_xr-E\cdot\nabla_vr)&=0. \label{eq:transport_step_2}
    \end{align}
\end{subequations}

Let $\Delta t$ denote the time step size and $(r^n, j^n)$ denote the solution at  $n$-th time level. The temporal scheme for advancing ($r^n, j^n$) to ($r^{n+1}, j^{n+1}$) consists of schemes for updating ($r^n, j^n$) to the intermediate solution  ($r^{\ast}, j^{\ast}$) by the  relaxation step and for updating ($r^{\ast}, j^{\ast}$) to ($r^{n+1}, j^{n+1}$) by the transport step.

For the relaxation step \eqref{eq:relaxation_step} which involves the challenge arising from  the stiffness of the collision term,  we adopt a robust explicitly implementable implicit scheme  proposed in \cite{jin_discretization_2000}, where the variable $r$ is discretized by the following first-order time-relaxed scheme 
\begin{equation}\label{eq:r_discrete}
    r^*=(1-\tau)r^n+\tau(1-\tau)\frac{\mathcal{P}(r^n)}{\mu}+\tau^2\rho^nM,
\end{equation}
with  $\tau = 1-\exp (-\mu\, \Delta t/\varepsilon^2)$ and $\mathcal{P}(r) = \mathcal{Q}(r)+\mu r$ being a non-negative operator, and the variable $j$ is discretized by the backward Euler scheme 
\begin{equation}
\label{eq:j_discrete}
    j^{\ast}=\alpha j^n-\beta(v\cdot\nabla_x r^{\ast}-E\cdot\nabla_v r^{\ast}), 
\end{equation}
with the coefficients  defined as
\begin{equation*}
    \alpha=\frac{\varepsilon^2}{\varepsilon^2+\lambda \Delta t},\quad \beta=\frac{\Delta t(1-\varepsilon^2\phi)}{\varepsilon^2+\lambda \Delta t}. 
\end{equation*}
It is worth mentioning that the scheme \eqref{eq:r_discrete}-\eqref{eq:j_discrete} is asymptotic-preserving. We refer readers to \cite{jin_discretization_2000} for a comprehensive discussion of the technical intricacies.

For the transport step \eqref{eq:transport_step}, we adopt the third-order SSPRK method \cite{shu_total-variation-diminishing_1988}. This  time discretization is a convex combination of first order forward Euler steps. Therefore, to simplify the presentation, we use the following first-order forward Euler scheme 
\begin{subequations}
\label{eq:transport_discrete}
    \begin{align}
\frac{r^{n+1}-r^*}{\Delta t}&+v\cdot\nabla_{x}j^* - E\cdot\partial_vj^*=0,\label{eq:transport_discrete_r}\\
\frac{j^{n+1}-j^*}{\Delta t}&+\phi(v\cdot\nabla_{x}r^*-E\cdot\nabla_v r^*)=0,\label{eq:transport_discrete_j}
    \end{align}
    \end{subequations}
as an example to show the framework of the fully discrete scheme and theoretical results in the following sections. Other explicit solvers can also be used here.

\section{Fully discrete numerical scheme}
\label{sec:full}
In this section, we present the fully discrete scheme with velocity and space discretizations for the systems \eqref{eq:relaxation_step} and \eqref{eq:transport_step}. For the space discretization, we employ the DG method with a positive-preserving limiter, which will be discussed  in the following subsections. For the velocity discretization, we adopt the same approach in \cite{jin_discretization_2000}, and the details are provided in Appendix \ref{sec:v_discrete}. For simplicity, we only present the one-dimensional schemes in both spatial and velocity spaces. It is straightforward to generalize to the multi-dimensional cases.

\subsection{Discontinuous Galerkin method}
In this subsection, we apply the DG method for the spatial discretization of the temporal schemes \eqref{eq:r_discrete}, \eqref{eq:j_discrete} and \eqref{eq:transport_discrete}, while leaving the variable $v$ continuous in the discussions.

Let $\Omega_x=[x_L,x_R]$ be the computational domain. We consider a uniform partition of $\Omega_x$ into $N_x$ cells of size $h=(x_R-x_L)/N_x$.
Denote the cell as $ I_i=[x_{i-\frac{1}{2}}, x_{i+\frac{1}{2}}]$ and the cell center as $x_i=(x_{i+\frac12}+x_{i-\frac12})/2$, for $i=1, \dots, N$, where $ x_{i+\frac{1}{2}} = x_L+i h, \quad i=0, \dots, N.$ The finite element approximation space is defined by
\begin{equation*}
    \mathbb{V}^k_h=\{u\in L^2(\Omega): u|_{I_i}\in\mathbb{P}^k(I_i), \; i=1,\dots,N\},
\end{equation*}
where $\mathbb{P}^k(I_i)$ denotes the polynomial space of degree at most $k$ defined in the cell $I_i$. We also define the jump and average of $u$ at the cell interface $x_{i+\frac12}$ as 
\begin{align}
    \label{eq:jump_average}
    [u]_{i+\frac12}=u_{i+\frac12}^+-u_{i+\frac12}^-,\qquad \{u\}_{i+\frac12}=\frac12(u_{i+\frac12}^++u_{i+\frac12}^-), 
\end{align}
where $u_{i+\frac12}^+$ and $u_{i+\frac12}^-$ denote the left and right limits of the function $u$ at the cell interface $x_{i+\frac12}$, respectively. We further define the following notations to simplify the presentation:
\begin{equation*}
    (u,w)_i=\int_{I_i}uw\,dx, \qquad (u,w) = \int_{\Omega_x}uw\,dx.
\end{equation*}

With a slight abuse of notations, the DG method with the time integrator \eqref{eq:j_discrete} and \eqref{eq:transport_discrete} is defined as follows:  find the unique solutions $j^{\ast},\ r^{n+1}, \  j^{n+1} \in \mathbb{V}^k_h$ such that for all test functions {$\xi, \zeta, \eta \in \mathbb{V}^k_h$}, 
\begin{subequations}\label{eq:dg_local}
\begin{align*}
&(j^{\ast},\eta)_i=\alpha(j^n,\eta)_i+\beta\left(v(r^*,\partial_x\eta)_i-v\hat r^*_{i+\frac{1}{2}}\eta^-_{i+\frac{1}{2}}+v\hat r^*_{i-\frac{1}{2}}\eta^+_{i-\frac{1}{2}}+(E\partial_v r^*,\eta)_i\right),\\
&\left(\frac{r^{n+1}-r^*}{\Delta t},\xi\right)_i+v\left(\tilde j^*_{i+\frac{1}{2}}\xi^-_{i+\frac{1}{2}}-\tilde j^*_{i-\frac{1}{2}}\xi^+_{i-\frac{1}{2}}-(j^*, \partial_x\xi)_i\right)-(E\partial_vj^*,\xi)_i = 0,\\
&\left(\frac{j^{n+1}-j^*}{\Delta t},\zeta\right)_i  +\phi v\left(\tilde r^*_{i+\frac{1}{2}}\zeta^-_{i+\frac{1}{2}}-\tilde r^*_{i-\frac{1}{2}}\zeta^+_{i-\frac{1}{2}}-(r^*, \partial_x\zeta)_i\right)-(\phi E\cdot\partial_vr^* ,\zeta)_i= 0,
\end{align*}
\end{subequations}
holds for $i=1,\dots,N_x$, while the solution $r^*\in \mathbb{V}^k_h$ is obtained directly from the right hand side of \eqref{eq:r_discrete}. Here, $\hat r^*, \tilde j^*$ and $\tilde r^*$ are the so-called numerical fluxes, which are single valued functions defined at the cell interfaces and depending on the values of the numerical solutions from both sides of the interface. We choose the following alternating fluxes 
\begin{equation*}
\label{eq:flux}
    \hat r^*_{i+\frac{1}{2}}=r^+_{i+\frac{1}{2}},\quad \tilde j^*_{i+\frac{1}{2}}=j^-_{i+\frac{1}{2}},\quad \tilde r^*_{i+\frac{1}{2}}=r^-_{i+\frac{1}{2}}.
\end{equation*}

For the convenience of analysis, by summing up the above DG schemes over all cells, we obtain the DG method with the time integrator \eqref{eq:r_discrete}, \eqref{eq:j_discrete} and \eqref{eq:transport_discrete} in the global form for the relaxation step
\begin{subequations}\label{eq:dg_relaxation}
\begin{align}
r^{\ast}&=(1-\tau)r^n+\tau(1-\tau)\frac{\mathcal{P}(r^n)}{\mu}+\tau^2\rho^nM,\\
(j^{\ast},\eta)&=\alpha(j^n,\eta)+\beta\left(v\sum_i \hat r^*_{i-\frac{1}{2}}[\eta]_{i-\frac{1}{2}}+v(r^*,\partial_x\eta)+(E\partial_v r^*,\eta)\right),
\end{align}
\end{subequations}
and for the transport step
    \begin{subequations}\label{eq:dg_transport}
    \begin{align}
        \left(\frac{r^{n+1}-r^*}{\Delta t},\xi\right) &- v \left(\sum_i\tilde j^*_{i-\frac{1}{2}}[\xi]_{i-\frac{1}{2}}+(j^*, \partial_x\xi)\right)-(E\partial_vj^*,\xi) = 0,\\
        \left(\frac{j^{n+1}-j^*}{\Delta t},\zeta\right) &- \phi v \left(\sum_i\tilde r^*_{i-\frac{1}{2}}[\zeta]_{i-\frac{1}{2}}+(r^*, \partial_x\zeta)\right)-(\phi E\partial_vr^* ,\zeta)= 0.
    \end{align}
    \end{subequations}

\subsection{Positivity-preserving limiter}

In this section, we apply a positivity-preserving limiter to ensure that the obtained numerical solution $f^n$ at each time step is non-negative, which satisfies the physical property of the distribution function. For more details regarding this type of limiters, see the review paper \cite{zhang_maximum-principle-satisfying_2011}.


Starting from the numerical solution $f^n\in \mathbb{V}_h^k$ at time level $n$, which can be obtained by $f^n=r^n+\varepsilon j^n$ according to \eqref{eq:def_even-odd}, a positivity-preserving limiter is applied to ``limit" $f^n$ to obtain a new function  $f^{n,new}\in \mathbb{V}_h^k$, which is then advanced to the next time level. The limiting procedure to construct $f^{n,new}$ from $f^n$ is as follows.

For simplicity, we omit the superscript $n$ and denote $f_i$ as the numerical solution in the cell $I_i$ and $\bar f_i$ as the cell average of $f_i$ in the cell $I_i$. The ``limited" function $f_i^{new}$ is constructed as
\begin{equation}\label{eq:limiter}
    f_i^{new} = \theta(f_i-\bar f_i)+\bar f_i,
\end{equation}
where $\theta$ is determined by
\begin{equation}
    \theta = \min \left(\frac{\bar f_{i}}{\bar f_i - f_\mathrm{min}},\; 1\right),\qquad f_\mathrm{min} = \min_{x\in I_i} f_i.
\end{equation}
We can easily conclude from Equation \eqref{eq:limiter} that the limiter preserves the cell average, i.e. $\bar f_j = \bar f_j^{new}$, and guarantee the positivity of $f^{new}_j$ if $\bar f_j \ge 0$.

For the parity variables $r$ and $j$ in the numerical schemes \eqref{eq:dg_relaxation} and \eqref{eq:dg_transport}, it follows from their definition in \eqref{eq:def_even-odd} that the limited functions $r_i^{new}$ and $j_i^{new}$ in the cell $I_i$ are constructed as
\begin{subequations}\label{eq:even-odd-limited}
 \begin{align}
        r_i^{new}(v)&=\frac12(f_i^{new}(v)+f_i^{new}(-v)),\\
        j_i^{new}(v)&=\frac1{2\varepsilon}(f_i^{new}(v)-f^{new}(-v)).
\end{align}   
\end{subequations}



\subsection{Boundary conditions}
In this subsection, we show how to deal with the boundary conditions in the DG framework. We restrict our attention to the case of inflow boundary conditions, which are considered in our numerical examples. Other types of boundary conditions can be handled in a similar way.

The incoming boundary conditions for the model problem \eqref{eq:Boltzmann} is given by 
\begin{equation}
    f(t,x_L,v) = F_L(v), \quad f(t,x_R,-v)=F_R(v), \qquad v>0,
\end{equation}
where $F_L$ and $F_R$ are assigned nonnegative functions. Similar to \cite{jin_discretization_2000},  the approximation of the boundary conditions for $r$ up to $O(\varepsilon^2)$ is represented as
\begin{align}\label{eq:boundary_condition_r}
\begin{split}
   r-\frac{\varepsilon}{\lambda}(v\partial_xr-E\partial_vF_L)\Big|_{x=x_L}=F_L,\\
    r+\frac{\varepsilon}{\lambda}(v\partial_xr-E\partial_vF_R)\Big|_{x=x_R}=F_R, 
\end{split}
\end{align}
and for $\epsilon\ll 1$, a reasonable approximation for $j$ is given by 
\begin{equation}\label{eq:boundary_condition_j}
    j(v)=\frac{1}{\lambda}(-v\partial_xr+E\partial_vr),\quad x=x_L,x_R.
\end{equation}

By implementing these boundary conditions in the DG framework, the numerical fluxes $\hat{r}$ at the cell interfaces $x_{1/2}$ and $x_{N+1/2}$ is discretized as
\begin{subequations}
    \begin{align}
        \hat{r}_{1/2}&=\frac{h(\lambda F_L-\varepsilon E_{1/2}\, \partial_v F_L)+2\varepsilon v r_1}{\lambda h+2\varepsilon v},\label{eq:boundary_flux_1}\\[4pt]
        \hat{r}_{N+1/2}&=\frac{h(\lambda F_R+\varepsilon E_{N+1/2}\,\partial_v F_R)+2\varepsilon v r_N }{\lambda h+2\varepsilon v},\label{eq:boundary_flux_2}
    \end{align}
\end{subequations}
where  $r_1=r(x_{1})$ and $r_N=r(x_{N})$ are the values of the DG solution $r$ at the centers of the cells  $I_1$ and $I_N$, respectively.
Similarly, the numerical fluxes $\hat{j}$ at the cell interfaces $x_{1/2}$ and $x_{N+1/2}$ is
 \begin{subequations}
\begin{align}
     \hat{j}_{1/2}&=\frac{1}{\lambda}\left(-v\frac{r_1-\hat{r}_{1/2}}{h/2}+E_{1/2}\partial_vF_L\right),\\
     \hat{j}_{N+1/2}&=\frac{1}{\lambda}\left(-v\frac{r_N-\hat{r}_{N+1/2}}{h/2}+E_{N+1/2}\partial_vF_R\right). 
 \end{align}
 \end{subequations}


\section{Theoretical analysis}
\label{sec:3}
In this section, we investigate the stability and AP property of the proposed numerical scheme discussed in Section \ref{sec:full}. 
For simplicity, we ignore the effect of the external field $E$ and consider the case when $\varepsilon < 1$, which implies that $\phi=1$. Then the schemes \eqref{eq:dg_relaxation}-\eqref{eq:dg_transport} with numerical fluxes \eqref{eq:flux} can be  reformulated as
\begin{subequations}\label{eq:simple_4}
\begin{align}
    r^*=(1-\tau)r^n+\tau(1-\tau)\frac{\mathcal{P}(r^n)}{\mu}+\tau^2\rho^nM,\label{eq:simple_4_relax_r}\\
    \left(\frac{j^*-j^n}{\Delta t},\eta\right)=-\frac{\lambda}{\varepsilon^2}(j^*,\eta)
    -\frac{1-\varepsilon^2}{\varepsilon^2} v \mathcal{L}^+(r^*,\eta),\label{eq:simple_4_relax_j}\\
    \left(\frac{r^{n+1}-r^*}{\Delta t},\xi\right) + v \mathcal{L}^-(j^*,\xi) = 0,\label{eq:simple_4_trans_r}\\
    \left(\frac{j^{n+1}-j^*}{\Delta t},\zeta\right)+ v \mathcal{L}^-(r^*,\zeta)=0,\label{eq:simple_4_trans_j}
\end{align}
\end{subequations}
where the discretization operators $\mathcal{L}^\pm$ are  defined as
\begin{equation}\label{eq:discrete_operator}
        \mathcal{L}^\pm(\psi,u)=-(\psi,u_x)-\sum_i\psi^\pm_{i-\frac{1}{2}}[u]_{i-\frac{1}{2}},\quad \psi,\,u\in \mathbb{V}_h^k.
\end{equation}

\subsection{Stability analysis}
\label{sec:stability}
In this section, we carry out the stability analysis of the APDG scheme \eqref{eq:simple_4} for the model problem with $\sigma$ being constant.  Periodic boundary conditions are applied. This section is organized as follows: In Section \ref{sec:pre}, we introduce the norms and inequalities that will be used in the proof. In Section \ref{sec:main}, we present the main result and the proof.
\subsubsection{Preliminaries}
\label{sec:pre}
We adopt the weighted $L^2$ norm  for the velocity variable $v$, given by
\begin{equation}
    \|f\|_{L^2(dv), \mathrm{weighted}} = \left<f^2\right>^{\frac12} = \left(\int_\mathbb{R}f^2\frac{\lambda}{M}\,dv\right)^{\frac12},
\end{equation}
with the operator $\left< \cdot \right>$ defined as
\begin{equation}
\label{eq:weightnorm}
    \left<f\right>=\int_\mathbb{R}f\frac{\lambda}{M}\,dv.
\end{equation}
This weighted norm was introduced in \cite{poupaud1991}, where the operator $\mathcal{Q}$ is proved  bounded under this norm. We use the standard $L^2$ norm  for the space variable $x$ 
and further define the norm in the phase space as 
\begin{equation}
\label{eq:phasenorm}
    \sanshu f\sanshu =\left<\|f\|_{L^2(dx)}^2\right>^\frac{1}{2}=\left(\int_{\Omega_x}\int_\mathbb{R}f^2\frac{\lambda}{M}\,dvdx\right)^{\frac12}.
\end{equation}

We state  the classical inverse properties for the finite element space $\mathbb{V}_h^k$ in the following lemma, and refer readers to \cite{ciarlet_finite_2002} for the proof.
\begin{lemma}[inverse inequality]
\label{lemma:inverse}
For any functions $u\in\mathbb{V}_h^k$, there holds
\begin{subequations}
\begin{align}
  \|u_x\|^2_{L^2(I_i)}&\leq C_i h^{-2}\|u\|^2_{L^2(I_i)},\\[2mm]
    \big|u(x_{i\pm\frac12})\big|^2&\leq C_t h \|u\|^2_{L^2(I_i)},
\end{align}
\end{subequations}
for $i=1,\dots, N_x$, where $C_t, C_i$ are positive constants independent of $u$, $I_i$, and $h$. 
\end{lemma}

The following lemma presents the properties of the operators $\mathcal{L}^\pm$ defined in \eqref{eq:discrete_operator}. The proof is straight forward by the definition of the operator and the application of periodic boundary conditions. Hence, the proof is omitted here. 
\begin{lemma}
	\label{lemma:lproperty}
	For any $u,\varphi\in \mathbb{V}_h^k$, there holds the equality
	\begin{align}
		\label{eq:lskew}
	\mathcal{L}^+(u,\varphi)+\mathcal{L}^-(\varphi,u)=0.
	\end{align}	
\end{lemma}

\subsubsection{Main result}
\label{sec:main}
We state the main result in the  theorem.
\begin{theorem}[Stability]
\label{Thm:main}
Suppose $\varepsilon <1$, $E=0$, and $\sigma$ is a constant in the model problem with periodic boundary conditions.  When the APDG scheme \eqref{eq:simple_4} is applied to the system, the numerical solution $r^n, j^n$ satisfies 
    \begin{equation}
        \sanshu r^{n+1}\sanshu^2+{\varepsilon^2}\sanshu j^{n+1}\sanshu^2\leq
        \sanshu r^{n}\sanshu^2+{\varepsilon^2}\sanshu j^n\sanshu^2,
    \end{equation}
    under the CFL condition 
    \begin{equation}
        \Delta t\leq \min\left\{\frac{\lambda h^2}{(1-\varepsilon^2)(C_i+4C_t^2)(2N_v+1)},\ C_0h\right\},
    \end{equation}
    where $C_i, C_t, C_0$ are positive constants independent of $h$ and $\Delta t$.
\end{theorem}

\begin{remark}
    One can get a uniform stability respect to $\varepsilon$, with the CFL condition tends to be $\Delta t=O(h^2)$ as $\varepsilon \ll 1$.
\end{remark}


\begin{lemma}
\label{lemma:r*}
Assume $\sigma$ is a constant. Then the numerical solution $r^*$ of the scheme \eqref{eq:simple_4} satisfies
\begin{equation}\label{eq:thm1}
\sanshu r^*\sanshu \leq \sanshu r^n\sanshu 
\end{equation}
\end{lemma}
  
\begin{proof}
It follows from the assumption that $\sigma$ is a constant that $\lambda=\sigma$, and $\mathcal{P}(r)=\lambda \rho M+(\mu-\lambda) r$. Then the scheme \eqref{eq:simple_4_relax_r} can be rewritten as
\begin{equation}\label{eq:r_delta}
    r^*=\delta r^n+\left(1-\delta\right)\rho^nM,
\end{equation}
where $\delta=\left(1-\tau^2-\frac{\lambda}{\mu}\tau(1-\tau)\right).$

By taking the norm $\sanshu \cdot \sanshu $ on both sides of Equation \eqref{eq:r_delta},
we obtain
    \begin{equation}
        \sanshu r^*\sanshu \leq \delta \sanshu r^n\sanshu +(1-\delta)\sanshu \rho^nM\sanshu .
    \end{equation}
Thus, the conclusion \eqref{eq:thm1} can be proved by showing that
    \begin{equation}\label{eq:sufficient}
        \sanshu \rho^nM\sanshu \leq\sanshu r^n\sanshu ,
    \end{equation}
which will be discussed in the following.    Without causing any ambiguity, we omit the superscript $n$. By the Cauchy-Schwarz inequality, we have
    \begin{equation}\label{eq:cauchy_schwarz}
        \int_{\mathbb{R}}\rho^2M^2\frac{\lambda}{M}\,dv=\lambda\left( \int_\mathbb{R} r \, dv \right)^2 \leq\lambda\int_\mathbb{R} \frac{r^2}{M} \, dv  \int_\mathbb{R} M \,dv=\int_{\mathbb{R}}r^2\frac{\lambda}{M}\,dv.
    \end{equation}
    Then the proof is complete by integrating both sides of \eqref{eq:cauchy_schwarz} over $x$. 
\end{proof}

\begin{lemma}\label{thm:step1}
Under the same assumption as in Theorem \ref{Thm:main}, the numerical solution $r^{n+1}, j^*, r^*, j^n$ of the APDG scheme \eqref{eq:simple_4} satisfies 
    \begin{equation}\label{eq:thm2}
        \sanshu r^{n+1}\sanshu^2+\frac{\varepsilon^2}{1-\varepsilon^2}\sanshu j^*\sanshu^2\leq\sanshu r^{*}\sanshu^2+\frac{\varepsilon^2}{1-\varepsilon^2}\sanshu j^n\sanshu^2
    \end{equation}
    under the CFL condition
    \begin{equation}\label{CFL}
        \frac{\Delta t}{h^2}\leq \frac{\lambda}{(1-\varepsilon^2)(C_i+4C_t^2)(2N_v+1)},
    \end{equation}
where $C_i, C_t$ are positive constants independent of $\Delta t$, $h$ and the solution.
\end{lemma}
\begin{proof}
  By taking $\eta=\frac{\varepsilon^2}{1-\varepsilon^2}j^*$ in \eqref{eq:simple_4_relax_j} and $\xi=r^{n+1}$ in \eqref{eq:simple_4_trans_r}, 
  we have
  \begin{subequations}\label{eq:simple_4_1}
\begin{align}
    \frac{\varepsilon^2}{\Delta t(1-\varepsilon^2)}\left(j^*-j^n,j^*\right)&=-\frac{\lambda}{1-\varepsilon^2}(j^*,j^*)
    - v \mathcal{L}^+(r^*,j^*),\label{eq:simple_4_relax_j_1}\\
    \frac{1}{\Delta t}\left(r^{n+1}-r^*,r^{n+1}\right) &+ v \mathcal{L}^-(j^*,r^{n+1}) = 0.\label{eq:simple_4_trans_r_1}
\end{align}
\end{subequations}
By adding \eqref{eq:simple_4_relax_j_1} and \eqref{eq:simple_4_trans_r_1} together, combining with the properties of $\mathcal{L}^\pm$ in Lemma \ref{lemma:lproperty} and  the fact
  \begin{align*}
     \left(j^*-j^n,j^*\right)&=\frac12 \left(\|j^*\|^2-\|j^n\|^2+\|j^*-j^n\|^2\right),\\
      \left(r^{n+1}-r^*,r^{n+1}\right) &= \frac{1}{2}\left(\|r^{n+1}\|^2-\|r^*\|^2+\|r^{n+1}-r^*\|^2\right),
  \end{align*}
and applying the operator $\left<\cdot\right>$ defined in \eqref{eq:weightnorm}, 
 we obtain
    \begin{equation}\label{eq:rn+1j}
        \begin{aligned}
            &\frac{1}{2\Delta t}\left(\sanshu r^{n+1}\sanshu^2+\frac{\varepsilon^2}{1-\varepsilon^2}\sanshu j^*\sanshu^2-\sanshu r^{*}\sanshu^2-\frac{\varepsilon^2}{1-\varepsilon^2}\sanshu j^n\sanshu^2\right)\\
            &=-\frac{1}{2\Delta t}\sanshu r^{n+1}-r^*\sanshu^2-\frac{\varepsilon^2}{2\Delta t(1-\varepsilon^2)}\sanshu j^*-j^n\sanshu^2\\
            &-\frac{\lambda}{1-\varepsilon^2}\sanshu j^*\sanshu^2
            +\left<v\mathcal{L}^+(r^{n+1}-r^*,j^*)\right>.
        \end{aligned}
    \end{equation}
    
    Now we  estimate the term $\left<v \mathcal{L}^+(r^{n+1}-r^*,j^*)\right>$, which, according to the definition \eqref{eq:discrete_operator}, can be rewritten as
    \begin{equation}\label{eq:estimate_vah}
        \begin{aligned}
         & |\left<v \mathcal{L}^+(r^{n+1}-r^*,j^*)\right>|=  |\left<v \mathcal{L}^+(\Delta r,j^*)\right>|\\&\leq \sum_i\int_{I_i}\left<\left|v \Delta r\partial_x j^*\right|\right>dx+
                \sum_i\left<\left|v\Delta r_{i-\frac12}^+[j^*]_{i-\frac12}\right|\right>\triangleq\Lambda_1+\Lambda_2,
        \end{aligned}
    \end{equation}
    where $\Delta r=r^{n+1}-r^*$ and $\Delta r_{i-\frac12}^+=r^{n+1,+}_{i-\frac12}-r^{*,+}_{i-\frac12}$. 
 
For the term $\Lambda_1$, a simple use of Young's inequality, together with Lemma \ref{lemma:inverse} and the definition of the norm in \eqref{eq:phasenorm}, yields
    \begin{equation}
    \label{eq:lambda1}
        \begin{aligned}
            \Lambda_1&\leq \theta_1\sum_i\int_{I_i}\left< \Delta r^2dx\right>+\frac{1}{4\theta_1}\sum_i\int_{I_i}\left<(v\partial_x j)^2\right>dx\\
            &\leq \theta_1 \sanshu \Delta r\sanshu^2+\frac{C_i}{4\theta_1h^2}\sanshu vj\sanshu^2,
        \end{aligned}
    \end{equation}
    where $\theta_1$ is an arbitrary positive constant. Note  $j=M\psi$ for the velocity discretization with $\psi$ defined in \eqref{Psi}. Clearly $\psi$ is an $N_v$-th polynomial in $v$. Thus $\psi^2$ is a $2N_v$-th polynomial in $v $, and can be rewritten as 
    \begin{equation*}
        \psi(v)^2=\sum_{\ell=0}^{2N_v}\tilde\psi_\ell v^\ell,
    \end{equation*}
   where the coefficients $\psi_\ell$ are linear combinations of the coefficients of the Hermite expansion \eqref{Psi}.  Then the term $\sanshu vj\sanshu $ can be estimated as
\begin{align}
\label{eq:vj}
\begin{split}
\sanshu vj\sanshu^2
&=\int_{\Omega_x}\sum_{\ell=0}^{2N_v}\tilde\psi_\ell\int_\mathbb{R} \lambda M v^{\ell+2}\,dv dx
=\int_{\Omega_x}\sum_{\ell=0}^{2N_v}(\ell+1)\tilde\psi_\ell\int_\mathbb{R} \lambda M v^{\ell}\,dv dx
\\
&\leq\int_{\Omega_x}(2N_v+1) \sum_{\ell=0}^{2N_v} \tilde\psi_\ell\int_\mathbb{R} \lambda M v^{\ell}\,dvdx
=(2N_v+1) \sanshu j\sanshu^2,
\end{split}
\end{align}
where we have used  the following equality, 
    \begin{equation*}
        \int_\mathbb{R} M v^{\ell+2}\,dv=(\ell+1)\int_\mathbb{R} M v^\ell\,dv, \quad \ell\geq0.
    \end{equation*}
By substituting \eqref{eq:vj} into \eqref{eq:lambda1}, we obtain
\begin{align}
    \label{eq:lambda1_approx}
       \Lambda_1\leq \theta_1 \sanshu \Delta r\sanshu^2+\frac{C_i(2N_v+1)}{4\theta_1h^2} \sanshu j\sanshu^2.
\end{align}

 For the term  $\Lambda_2$, similar estimate can be performed to obtain
    \begin{equation}\label{eq:lambda_2_approx}
        \begin{aligned}
            \Lambda_2&\leq \frac{\theta_2 h}{C_t}\sum_i\left<(\Delta r_{i-1/2}^+)^2\right>+\frac{C_t}{4\theta_2 h}\sum_i\left<(v[j]_{i-\frac12})^2\right>\\
            &\leq \frac{\theta_2 h}{C_t} \sum_i\left<(\Delta r_{i-1/2}^+)^2\right>+\frac{C_t}{2\theta_2 h}\sum_{i}\left(\left<(vj^{*,+}_{i-\frac12})^2+(vj^{*,-}_{i-\frac12})^2\right>\right)\\
            &\leq \theta_2\sanshu \Delta r\sanshu^2+\frac{C_t^2}{\theta_2 h^2}\sanshu vj\sanshu^2\\
            &\leq \theta_2\sanshu \Delta r\sanshu^2+\frac{C_t^2(2N_v+1)}{\theta_2 h^2}\sanshu vj\sanshu^2,
        \end{aligned}
    \end{equation}
    where $\theta_2$ is an arbitrary positive constant. 

  
It follows from substituting \eqref{eq:lambda1_approx} and \eqref{eq:lambda_2_approx} into  \eqref{eq:rn+1j} that
    \begin{equation}
        \begin{aligned}
            &\frac{1}{2\Delta t}\left(\sanshu r^{n+1}\sanshu^2+\frac{\varepsilon^2}{1-\varepsilon^2}\sanshu j^*\sanshu^2-\sanshu r^{*}\sanshu^2-\frac{\varepsilon^2}{1-\varepsilon^2}\sanshu j^n\sanshu^2\right)\\
            &\leq\left(\theta_1+\theta_2-\frac{1}{2\Delta t}\right)\sanshu r^{n+1}-r^*\sanshu^2 -\frac{\varepsilon^2}{2\Delta t(1-\varepsilon^2)}\sanshu j^*-j^n\sanshu^2 \\
            &+\left(\frac{C_i(2N_v+1)}{4\theta_1h^2}+\frac{C_t^2(2N_v+1)}{\theta_2h^2}-\frac{\lambda}{1-\varepsilon^2}\right)\sanshu j^*\sanshu^2,
        \end{aligned}
    \end{equation}
 which, by choosing $\theta_1=\theta_2=\frac{1}{4\Delta t}$ and dropping the positive term  $\frac{\varepsilon^2}{2\Delta t(1-\varepsilon^2)}\sanshu j^*-j^n\sanshu^2$, yileds
    \begin{equation*}
        \frac{1}{2\Delta t}\left(\sanshu r^{n+1}\sanshu^2+\frac{\varepsilon^2}{1-\varepsilon^2}\sanshu j^*\sanshu^2-\sanshu r^{*}\sanshu^2-\frac{\varepsilon^2}{1-\varepsilon^2}\sanshu j^n\sanshu^2\right)\leq
        \gamma \sanshu j^*\sanshu^2,
    \end{equation*}
    where the coefficient $\gamma$ is given by
    \begin{equation*}
        \gamma = \frac{\Delta t}{h^2}\left(C_i+4C_t^2\right)(2N_v+1)-\frac{\lambda}{1-\varepsilon^2}.
    \end{equation*}
Then the proof is complete under the CFL condition
    \begin{equation*}
        \frac{\Delta t}{h^2}\leq \frac{\lambda}{(1-\varepsilon^2)(C_i+4C_t^2)(2N_v+1)}.
    \end{equation*}

\end{proof}

\begin{lemma}
Under the same assumption as in Theorem \ref{Thm:main},
the numerical solutions $r^*$ ,$j^*$, $r^{n+1}$, $j^{n+1}$ 
of the APDG scheme with the third-order SSPRK temporal method is used in the transport step satisfy
    \begin{equation}\label{eq:thm3}
        \sanshu r^{n+1}\sanshu^2+\sanshu j^{n+1}\sanshu^2\leq
        \sanshu r^*\sanshu^2+\sanshu j^*\sanshu^2,
    \end{equation}
    under the 
    CFL condition
    \begin{equation*}
        \frac{\Delta t}{h} \leq C_0,
    \end{equation*}
    where $C_0$ is a positive constant independent of $\Delta t$ and $h$.
\end{lemma}
In fact, the transport step with $E=0$ is a system of linear conservations laws. Therefore, the proof of this lemma is similar to the proof of the strong stability for the DG method with the third-order SSPRK method solving linear conservation laws in  \cite{zhang_shu_stability}, and  is omitted here. 

\begin{proof}[Proof of Main Theorem \ref{Thm:main}]
  By multiplying both sides of Equation \eqref{eq:thm2} by $(1-\varepsilon^2)$ and Equation \eqref{eq:thm3} by $\varepsilon^2$, and adding them together, we have
    \begin{equation}
        \sanshu r^{n+1}\sanshu^2+\varepsilon^2\sanshu j^{n+1}\sanshu^2\leq
        \sanshu r^{*}\sanshu^2+\varepsilon^2\sanshu j^n\sanshu^2,
    \end{equation}
    which, combining with Lemma \ref{lemma:r*}, completes the proof.
    
\end{proof}

\subsection{Analysis of asymptotic-preserving property}
\label{sec:ap}
In this subsection, we  discuss the AP property of the proposed scheme.

It is straightforward to find that as $\varepsilon\to0$, $\tau\to1, \alpha\to0$, and $\beta\to\frac{1}{\lambda}$. Then the scheme \eqref{eq:simple_4} can be rewritten as follows:
\begin{subequations}
\label{eq:schemeap}
    \begin{align}
        r^*&=\rho^nM,\label{eq:ap_r*}\\
        (j^*,\eta)&=-\frac{v}{\lambda}\mathcal{L}^+(r^*,\eta),\label{eq:ap_j*}\\
        \left(\frac{r^{n+1}-r^*}{\Delta t},\xi\right)& + v\mathcal{L^-}(j^*,\xi) = 0,\label{eq:ap_rn+1}
    \end{align}
\end{subequations}
where we omit the scheme \eqref{eq:simple_4_trans_j} for updating from $j^*$ to $j^{n+1}$, since  $j^n$ is no longer needed for updating $j^*$ in \eqref{eq:ap_j*}, and consequently there is no need for updating $j^{n+1}$. 

We then introduce a new variable $g^*$ defined as
\begin{equation}\label{eq:ap_g*}
    j^*=-\frac{vM}{\lambda} g^*.
\end{equation}
By substituting \eqref{eq:ap_r*} and \eqref{eq:ap_g*} into \eqref{eq:ap_j*} and 
\eqref{eq:ap_rn+1}, and integrating with respect to $v$, we have
\begin{subequations}
\label{eq:LDG}
    \begin{align}
        (g^*,\eta)&=\mathcal{L}^+(\rho^n,\eta),\\
        \left(\frac{\rho^{n+1}-\rho^n}{\Delta t},\xi\right)&=D  \mathcal{L^-}(g^*,\xi),
    \end{align}
\end{subequations}
with the constant $D$  defined as
\begin{equation}
    D=\int_\mathbb{R}\frac{v^2M}{\lambda}\,dv.
\end{equation}

The limiting scheme \eqref{eq:LDG} is exactly the local DG (LDG) scheme with the forward Euler method for the following limiting diffusion equation
\begin{equation}
\partial_t \rho = D \partial_{xx} \rho,
\end{equation}
which demonstrates the AP property of the proposed scheme.
\section{Numerical examples}
\label{sec:4}
In this section, numerical examples are presented to demonstrate the performance of the proposed APDG scheme with the positive-preserving limiter. In particular, we will study its high-order accuracy in space and AP property,  and further consider a more complex problem with mixed regimes, where the Knudsen number $\varepsilon$ depends on the spatial variable. 


In all the examples, the scattering kernel is set as $\sigma=1$ rendering $\lambda=1$. Consequently, $\mu=2>\lambda$ is chosen in the relaxation step \eqref{eq:r_discrete}. 
The spatial domain is $\Omega_x=[0,1]$, and the number of velocity points is $N_v=15$. The APDG method adopts piecewise $\mathbb{P}^k$ polynomials with $k=2,3$ for Example \ref{eg:convergence}, and with $k=2$ for all other examples. 

\begin{exmp}[accuracy test]\label{eg:convergence}
    In this example, we consider the case with $E=0$ to test the accuracy of the APDG scheme in the spatial variable. The initial distribution is taken as $f(x,v,t=0)=M(v)(\cos(2\pi x)+1)$, and periodic boundary conditions in $x$ are applied. Both the diffusive regime with $\varepsilon=10^{-5}$ and  the kinetic regime with $\varepsilon=0.5$ are studied. 
\end{exmp}

When $\varepsilon\rightarrow 0$, the density $\rho$ computed from the semiconductor Boltzmann equation \eqref{eq:Boltzmann} converges to that of the limiting drift-diffusion equation. Specifically,  the exact solution of the drift-diffusion model with  $E=0$ is given by
\begin{equation*}
    \rho(x,t) = e^{-4{\pi^2}t}\cos(2\pi x)+1.
\end{equation*}
Thus, in the diffusive regime with $\varepsilon=10^{-5}$, we compute the errors between the exact solution $f_{exact}(x,v,t)={\rho}(x,t)M(v)$ and the numerical solution to test the accuracy.
In the kinetic regime with $\varepsilon=0.5$, where it is infeasible to get the exact solution, we  compute the error between the numerical solution with $N_x$ cells and with 2$N_x$ cells to test the order of accuracy


We simulate this example up to $t=0.03$ with a small time step size of $\Delta t=2\times10^{-6}$ to make the temporal error negligible compared to the spatial error.
Table \ref{tab:conv_test2}  lists  the $L^1$ and $L^\infty$ errors and orders of accuracy of $f$ for the APDG method using $\mathbb{P}^k$ polynomials with $k=2,3$. It can be observed that, in both the diffusive and kinetic regimes, the APDG method can achieve optimal $(k+1)$-th order of accuracy.
\begin{table}[!htbp]
    \centering
    \caption{Errors and orders of accuracy for the APDG method  in Example \ref{eg:convergence}.}
    \label{tab:conv_test2}
    \begin{tabular}{ccccccccc}
    \toprule
    &\multicolumn{4}{c}{$\varepsilon=10^{-5}$}&\multicolumn{4}{c}{$\varepsilon=0.5$}\\
    \cmidrule(lr){2-5}\cmidrule(lr){6-9}
    $N_x$ & $L^2$ error&order&$L^{\infty}$ error&order& $L^2$ error&order&$L^{\infty}$ error&order\\
    \midrule
    & $k=2$ & &&&&\\
    \cmidrule(lr){2-9}
    4&      4.14e-3&	-&	    1.62e-2&	-&      1.01e-1&	-&	    3.15e-1&	-\\
    8&	    5.12e-4&	3.02&	2.35e-3&	2.78&   1.01e-2&	3.32&	3.57e-2&	3.14\\
    16&	    6.39e-5&	3.00&	3.05e-4&	2.95&   1.22e-3&	3.05&	5.10e-3&	2.81\\
    32&	    7.99e-6&	2.99&	3.85e-5&	2.99&   1.52e-4&	3.01&	6.37e-4&	3.00\\
    64&     9.99e-7&	3.00&	4.82e-6&	3.00&   1.82e-5&	3.06&	7.71e-5&	3.05\\
   \midrule
    & $k=3$ & &&&&\\
    \cmidrule(lr){2-9}

    4&      3.90e-4&	-&	    1.69e-3&	-&      4.69e-3&	-&	    9.89e-3&	-\\
    8&      2.46e-5&	3.99&	1.31e-4&	3.69&   1.51e-3&	2.03&	4.70e-3&	1.07\\
    16&	    1.54e-6&	4.00&	8.52e-6&	3.94&   1.01e-4&	3.51&	4.76e-4&	3.30\\
    32&	    1.04e-7&	3.88&	5.97e-7&	3.84&   7.24e-6&	3.80&	3.38e-5&	3.82\\
    64&	    6.29e-9&	4.05&   3.46e-8&	4.11&   4.92e-7&	3.88&	2.23e-6&	3.92\\
    \bottomrule
    \end{tabular}
\end{table}

\begin{exmp}[Given electric potential]\label{eg:fixed_E}
In this example, we consider the electric field function given by
    \begin{equation*}
        E= -2c(1/4-x)\exp(-c(1/4-x)^2),
    \end{equation*}
with the constant $c=50\exp(1)$. The initial condition is $f(x,v,t=0)=M(v)$, and the boundary conditions are the incoming boundary conditions $F_L(v)=F_R(v)=M(v)$ for $v>0$. We set $\varepsilon=0.5$ and $\varepsilon=2\times10^{-3}$ to represent the kinetic and diffusive regimes, respectively.
\end{exmp}

We perform numerical simulations for the proposed APDG scheme  with the spatial and temporal steps taken as
$\Delta x=0.05$ and $\Delta t=10^{-5}$, respectively. We first plot the numerical solution of the density  $\rho$ at $t=0.5$ in Figure \ref{fig:t2_rho-x} with the reference solution obtained by the finite volume method \cite{jin_discretization_2000} with a fine mesh of $N_x=500$. It is evident that numerical results agree well with the reference solution.

In addition, we investigate the behavior of the numerical solution when $\epsilon\rightarrow 0$ by taking different values of $\varepsilon$ in the range of $[10^{-6}, 10^{-3}]$.
Figure \ref{fig:t2_diff} shows the errors between the numerical solution and the limiting drift-diffusion equation solution at $t=0.5$ with respect to $\varepsilon$.  We observe a first-order convergence rate in  $\varepsilon$, indicating that the proposed numerical scheme converges to the diffusive model in the order of $\varepsilon$. This phenomenon satisfies our expectations. 

For a more detailed visualization, we show the time evolution of the distribution function $f_h$ in both kinetic and diffusive regimes in Figure \ref{fig:t2_time_evolution}. It can be observed that, in the kinetic regime, the solution, as shown on the left side of Figure \ref{fig:t2_time_evolution}, varies fast and tends to the equilibrium state as time evolves from $t=0.05$ to $t=5$, while the distribution to the problem in the diffusive regime remains relatively steady, as seen on the right side of Figure \ref{fig:t2_time_evolution}. 

\begin{figure}[!htbp]
    \centering
    \subfigure[ Kinetic regime with $\varepsilon=0.5$]{\includegraphics[width=0.46\textwidth]{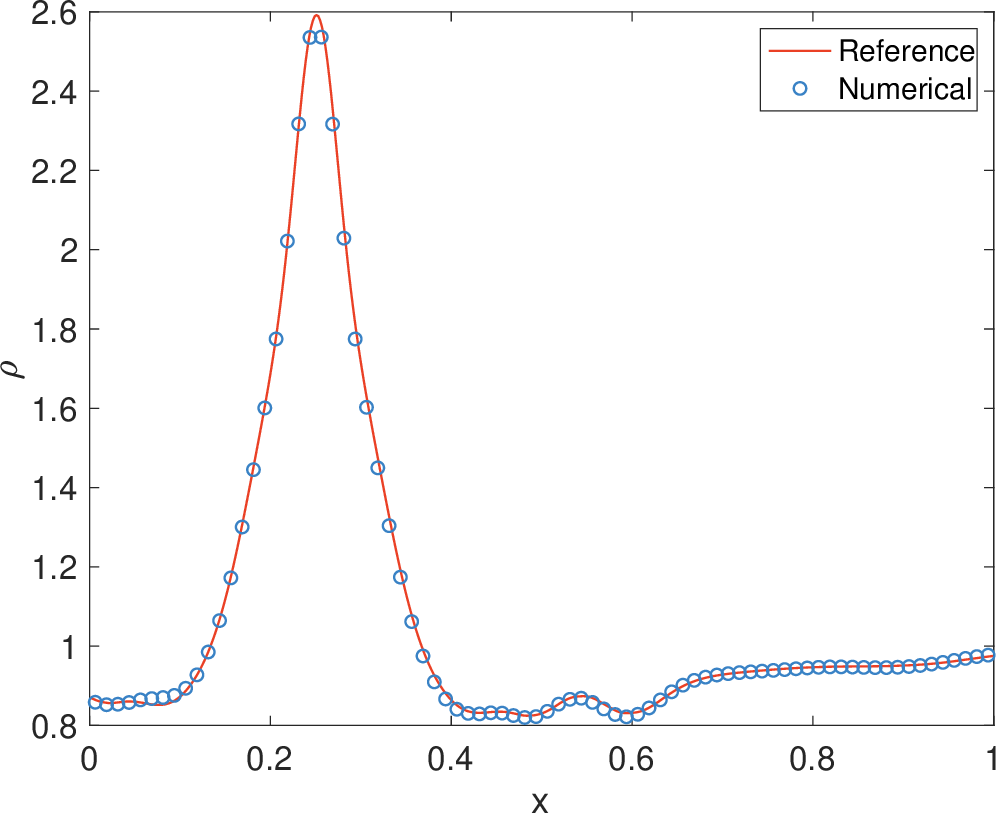}}\qquad
    \subfigure[ Diffusive regime with $\varepsilon=2\times10^{-3}$]{\includegraphics[width=0.445\textwidth]{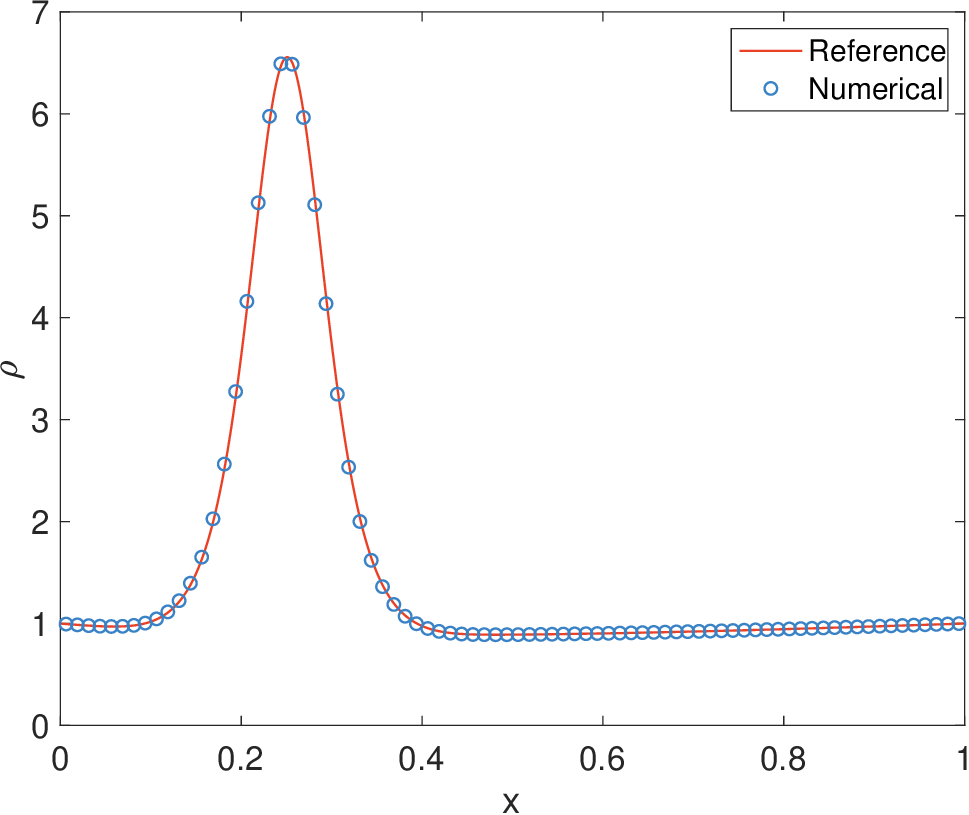}}
    \caption{Numerical solutions $\rho$ for the kinetic and diffusive regimes of Example \ref{eg:fixed_E}.}
    \label{fig:t2_rho-x}
\end{figure}

\begin{figure}[!htbp]
    \centering
    \includegraphics[width=0.5\textwidth]{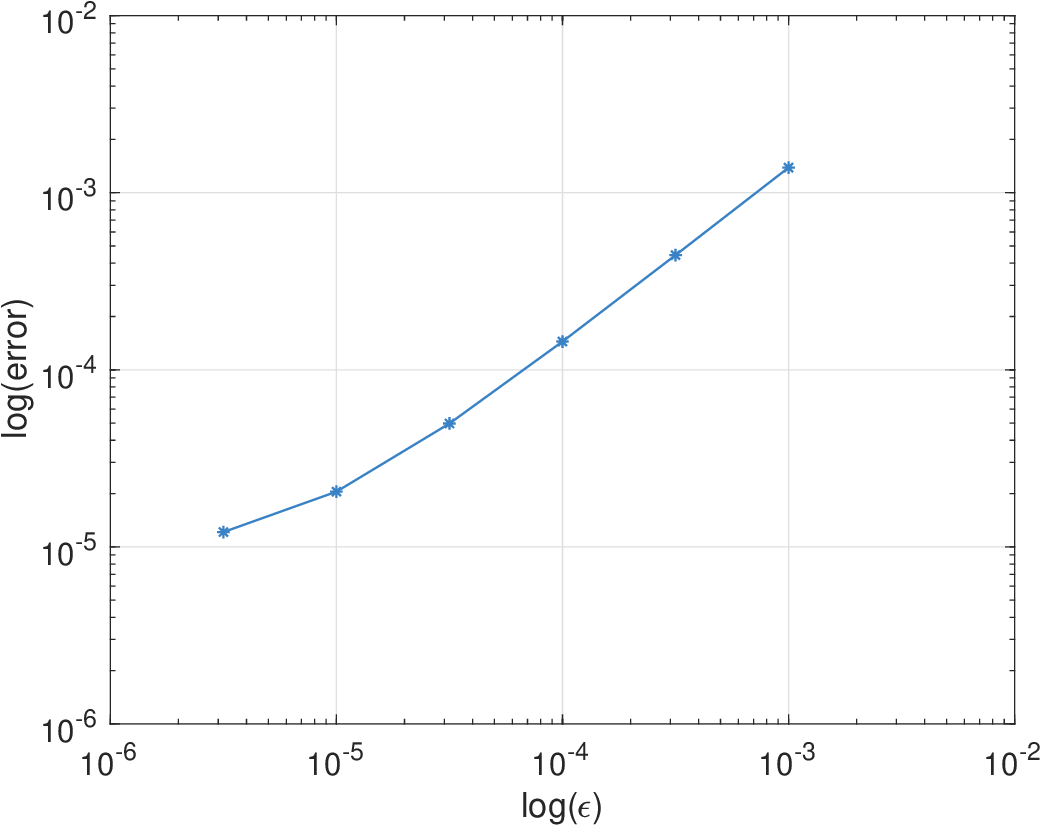}
    \caption{Convergence to the drift-diffusion system for Example \ref{eg:fixed_E}.}
    \label{fig:t2_diff}
\end{figure}

\begin{figure}[!htbp]
    \centering
    \subfigure[$\varepsilon=0.5, t=0.05$]{\includegraphics[width=0.48\textwidth]{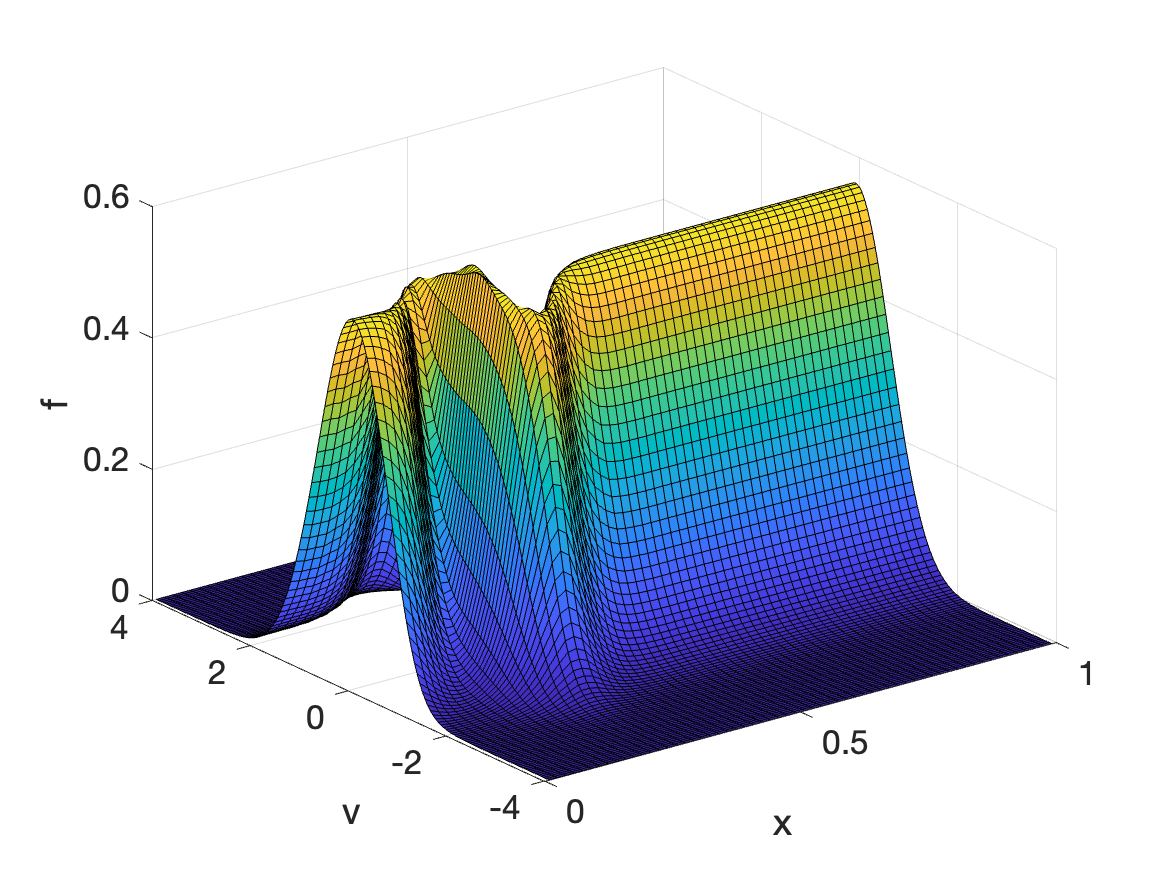}}
    \subfigure[$\varepsilon=2\times 10^{-3}, t=0.05$]{\includegraphics[width=0.48\textwidth]{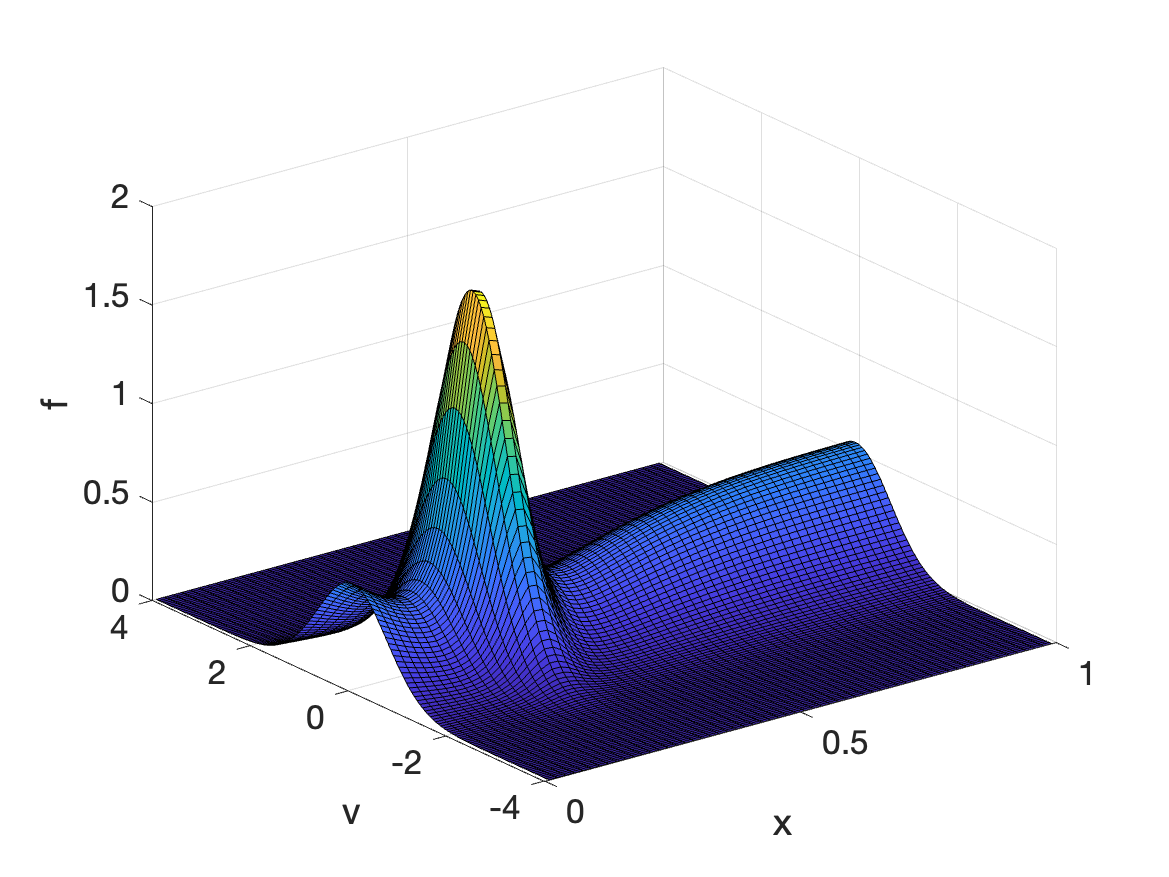}}
    \subfigure[$\varepsilon=0.5, t=0.5$]{\includegraphics[width=0.48\textwidth]{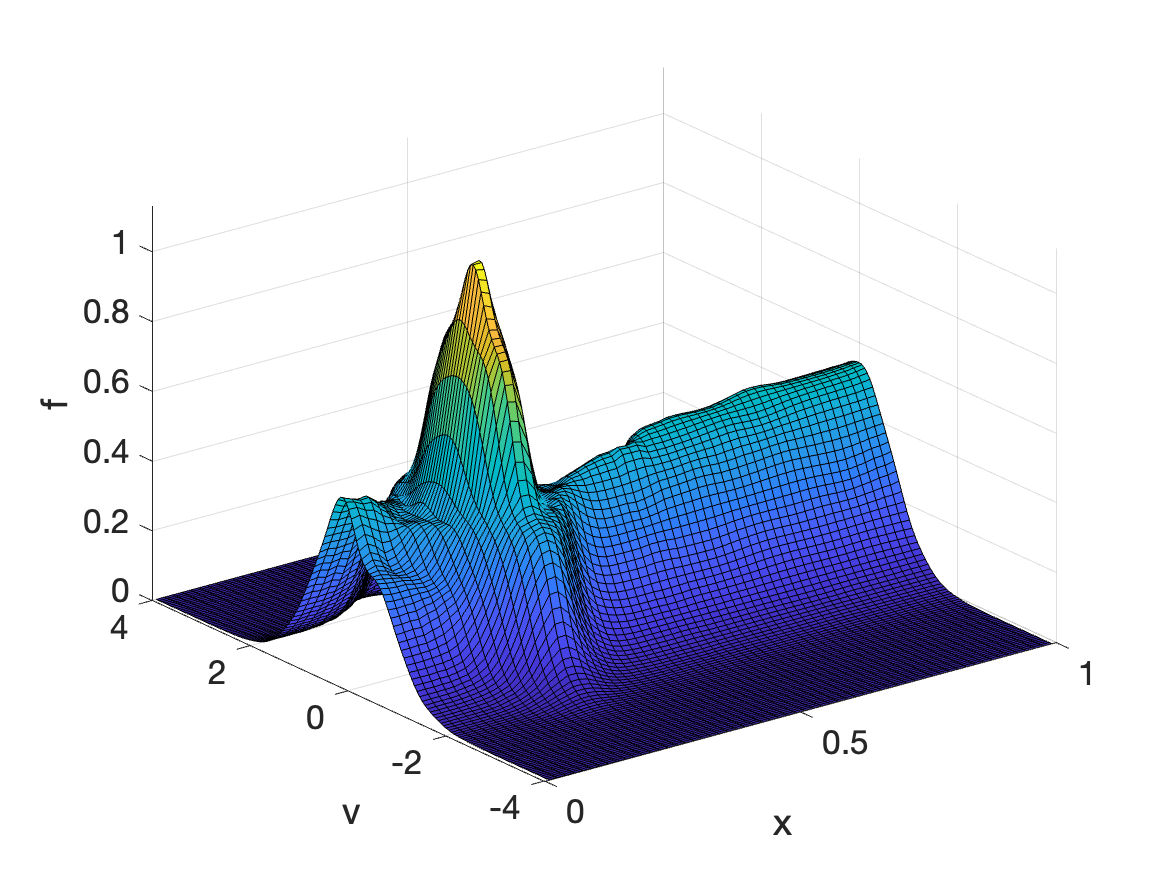}}
    \subfigure[$\varepsilon=2\times 10^{-3}, t=0.5$]{\includegraphics[width=0.48\textwidth]{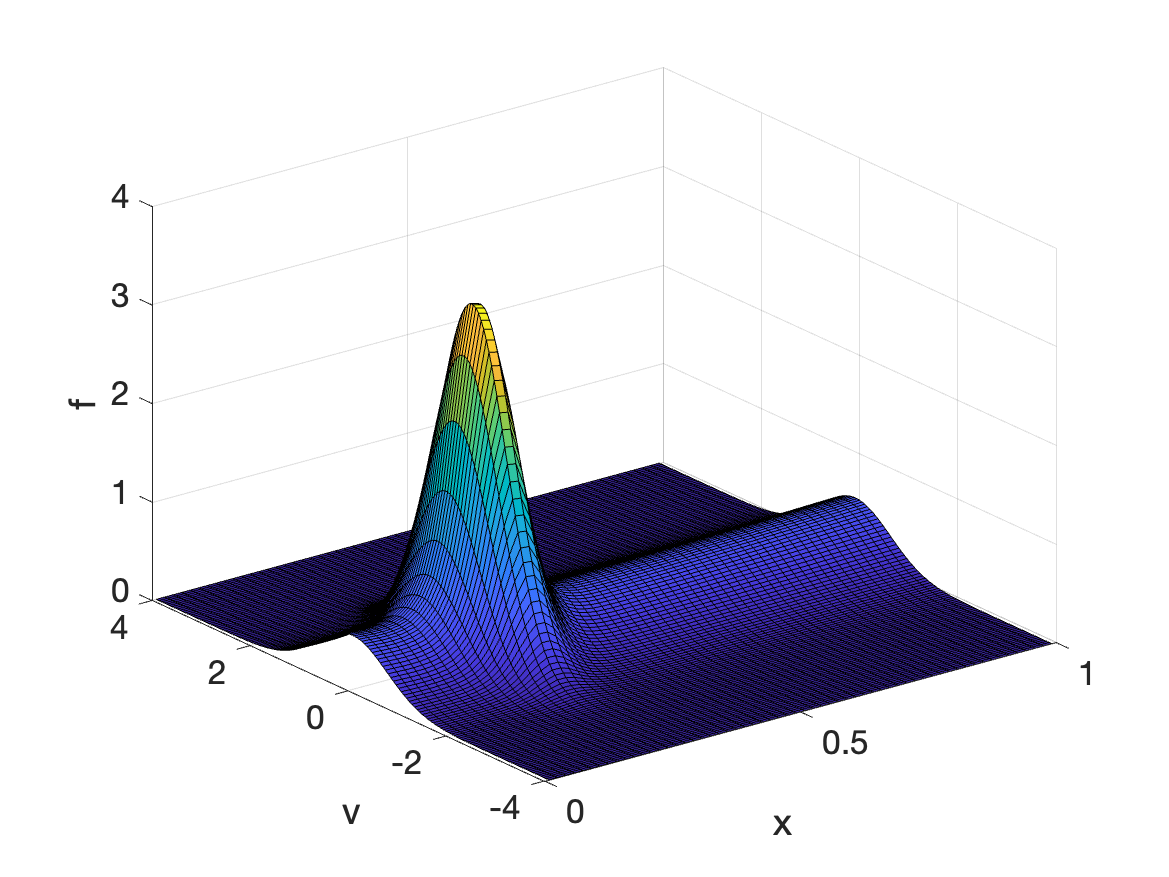}}
    \subfigure[$\varepsilon=0.5, t=5$]{\includegraphics[width=0.48\textwidth]{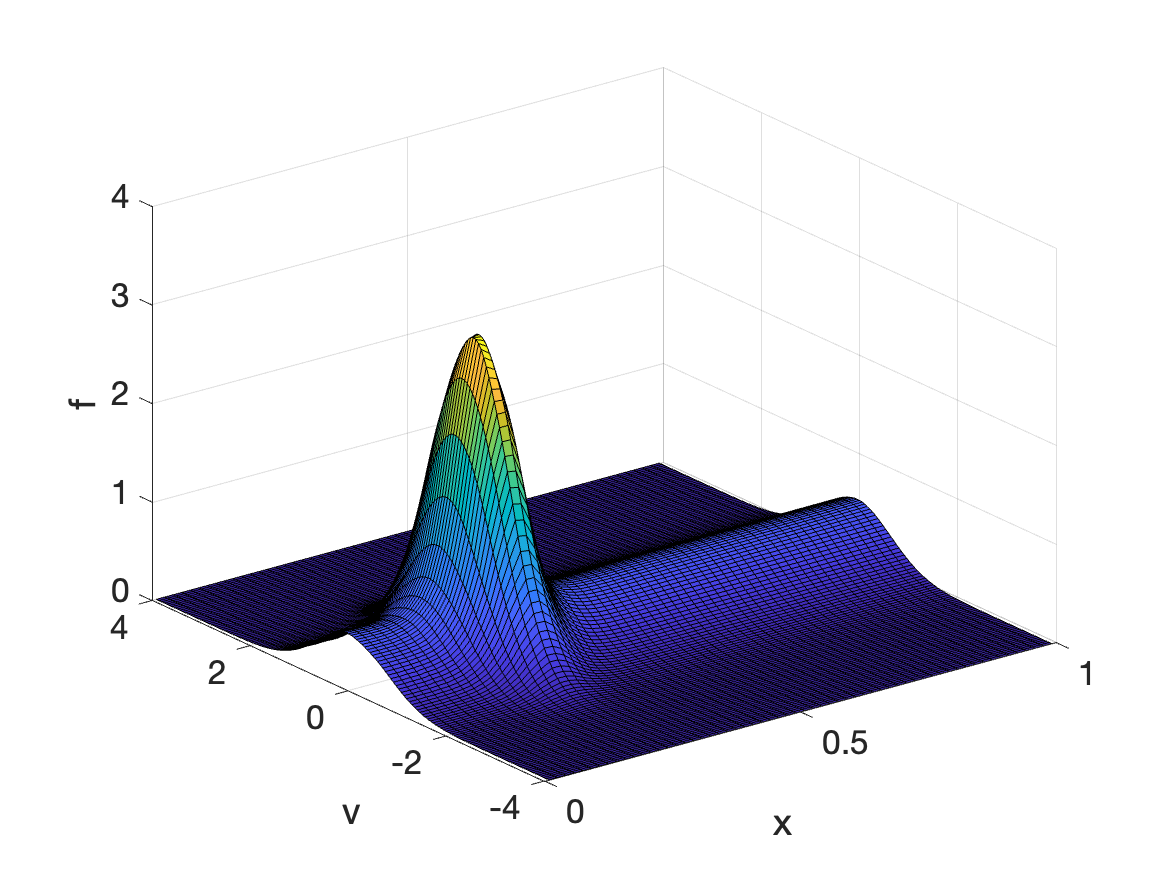}}
    \subfigure[$\varepsilon=2\times 10^{-3}, t=5$]{\includegraphics[width=0.48\textwidth]{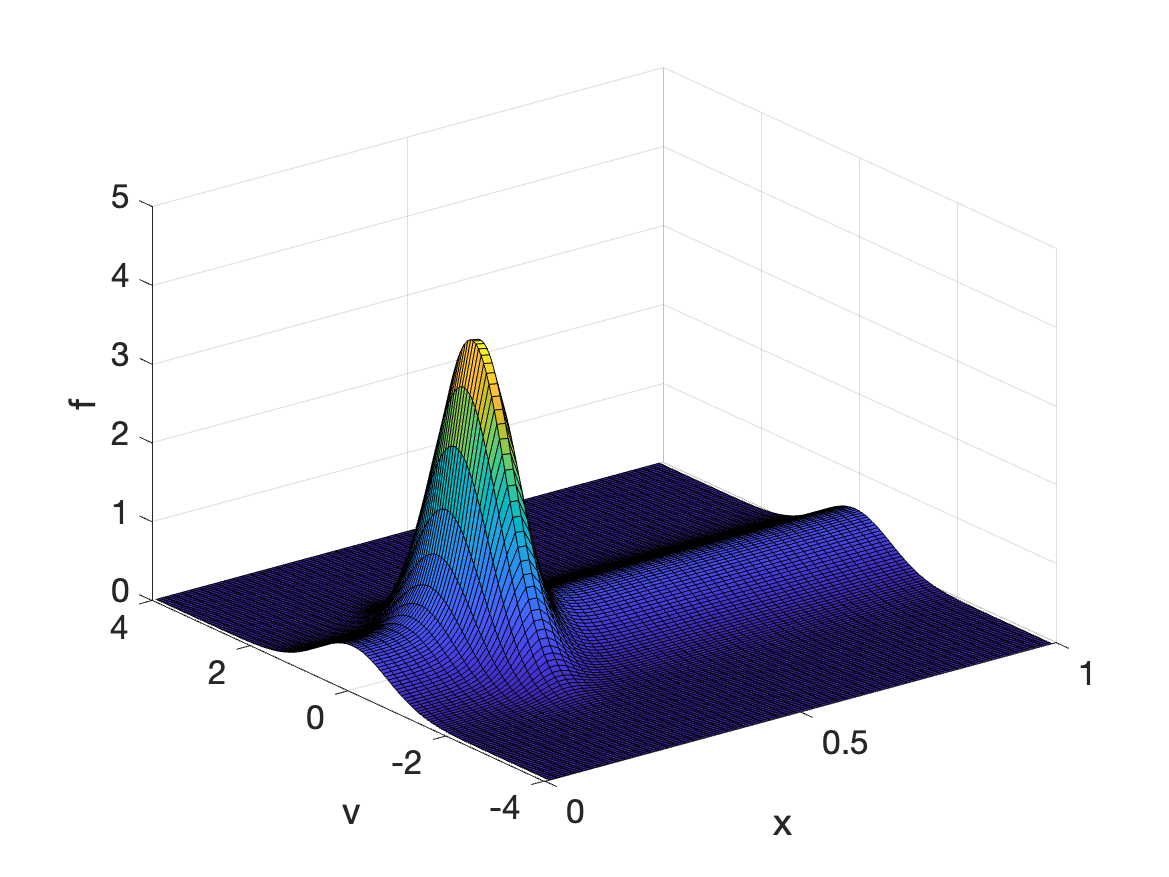}}
    \caption{Time evolution of numerical solutions $f_h$ for the kinetic (left) and diffusive (right) regimes of Example \ref{eg:fixed_E}, with $T=0.05$ (top), $T=0.5$ (middle), and $T=5$ (bottom).}
    \label{fig:t2_time_evolution}
\end{figure}

\begin{exmp}[Boltzmann-Poisson system]\label{eg:BPSys}
In this example, the electric potential and the electric field are given by the following Poisson equation
    \begin{gather*}
        \beta\Delta_x\Phi = \rho-c(x),\qquad E=-\nabla_x\Phi,\\
        \Phi(0) = 0,\qquad \Phi(1)=5,
    \end{gather*}
    where $\beta=0.002$ is the scaled Debye length and $c(x)$ is the doping profile written as
    \begin{equation*}
        c(x) = 1-(1-m)\left[\tanh\left(\frac{x-0.3}{s}\right)-\tanh\left(\frac{x-0.7}{s}\right)\right],
    \end{equation*}
    with $s=0.02, m=(1-0.001)/2$. The initial condition is $f(x,v,t=0)=M(v)$, with incoming boundary condition $F_L(v)=F_R(v)=M(v)$ for $v>0$. 
\end{exmp}

We simulate this example up to $t=0.05$ with $\Delta x=0.05$ and $\Delta t=2\times 10^{-5}$. 
We show in Figure \ref{fig:BPSys}(a) the density obtained from the numerical solution $f$ when $\varepsilon=10^{-3}$, comparing with that solved by the drift-diffusion equation. One can observe that the two solutions match well with each other, which indicates that our APDG scheme is able to accurately capture the diffusive behavior of the kinetic equation when $\varepsilon$ is small. Figure \ref{fig:BPSys}(b) presents the time evolution of $L^2$ errors between the distribution function $f$ and the corresponding local equilibrium state $f_{eq}=\int_{\mathbb R}f dv\, M(v)$. We can observe that, as time evolves, the distribution function converges to the local equilibrium state, as expected. 
We also show in Figure \ref{fig:BPSysEpotential} the electric field $E$ and the electric potential $\Phi$,  which match well with the counter parts of the drift-diffusion equation.

\begin{figure}[!htbp]
    \centering
    \subfigure[ $\rho$]{\includegraphics[width=0.42\textwidth]{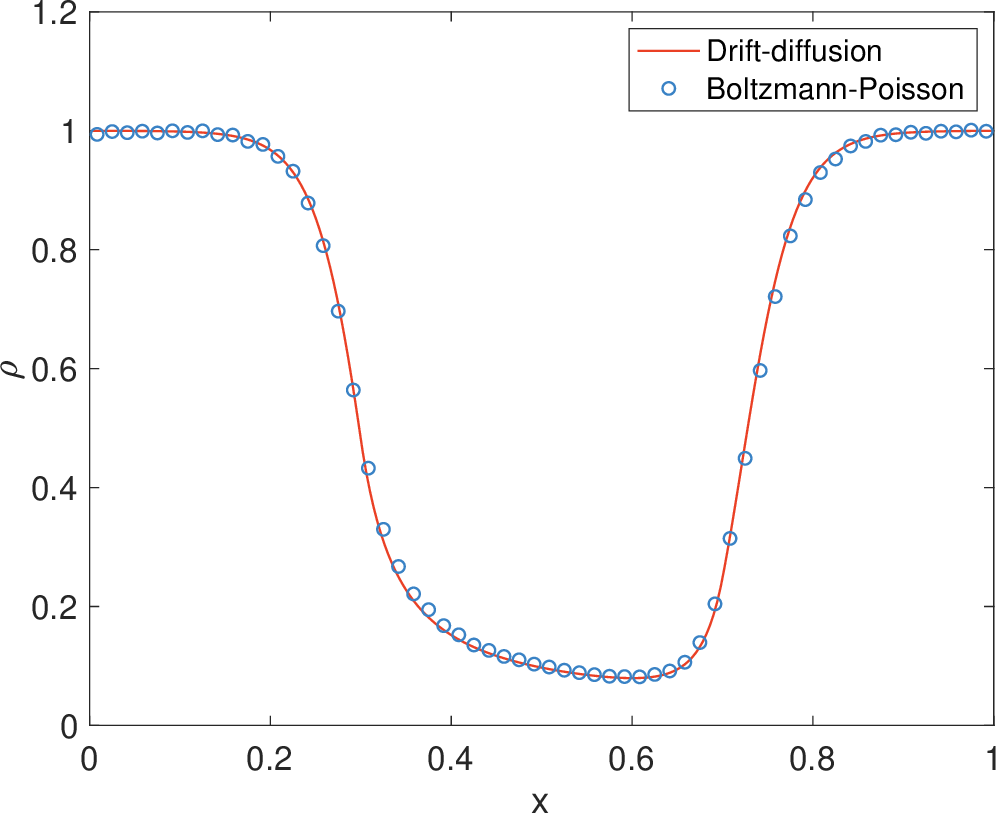}}\qquad\qquad
    \subfigure[Time evolution of $\sanshu f-f_{eq}\sanshu$]{\includegraphics[width=0.44\textwidth]{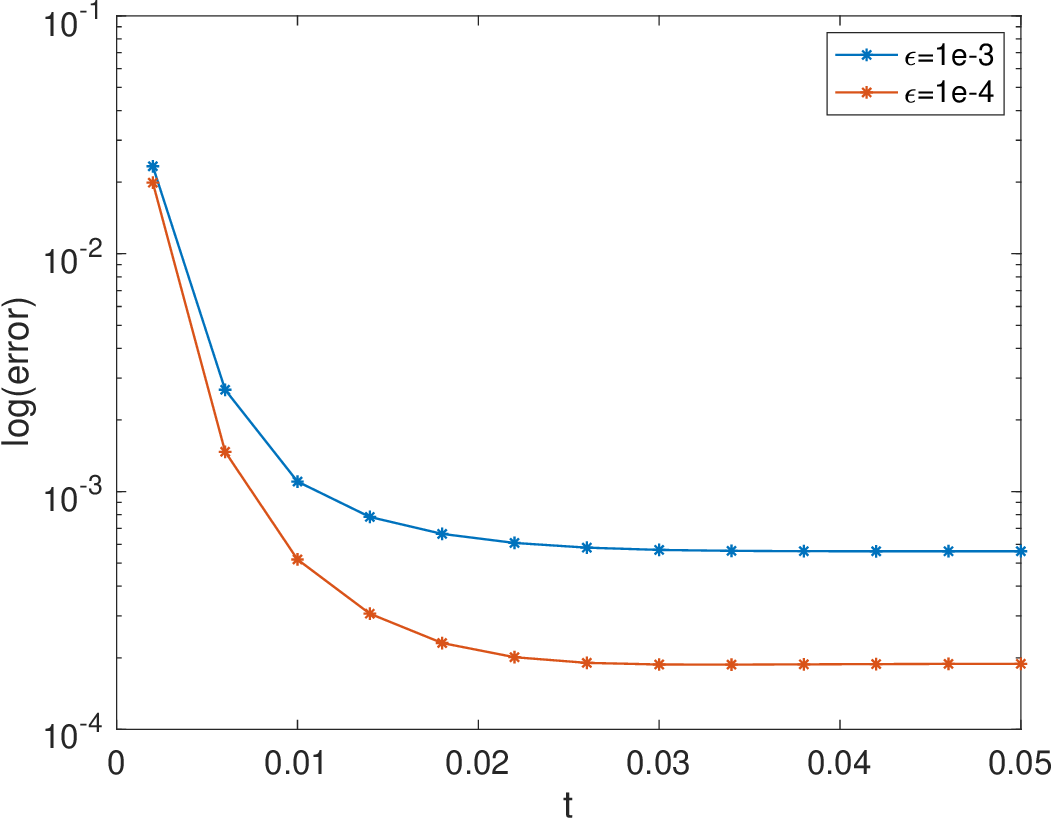}}
    \caption{Numerical results for the Boltzmann-Poisson system in Example \ref{eg:BPSys}. 
    }
    \label{fig:BPSys}
\end{figure}
\begin{figure}[!htbp]
\centering
\subfigure[Electric field $E$]{\includegraphics[width=0.42\textwidth]{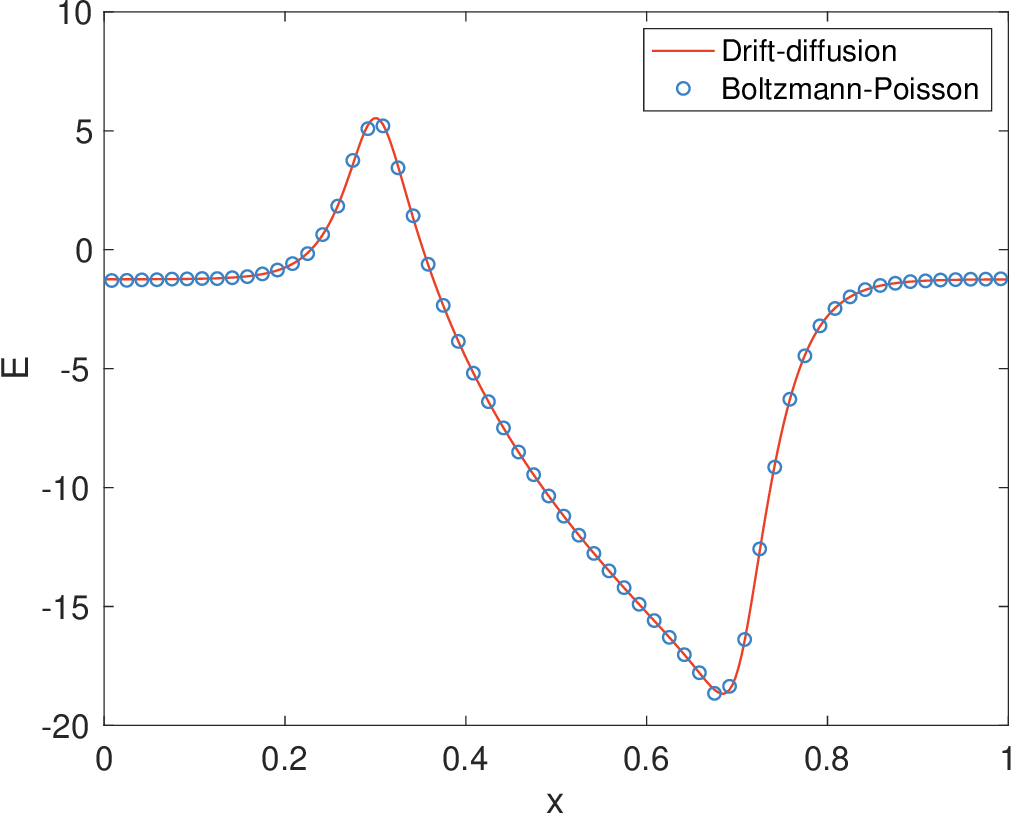}}\qquad\qquad
    \subfigure[Electric potential $\Phi$]{\includegraphics[width=0.42\textwidth]{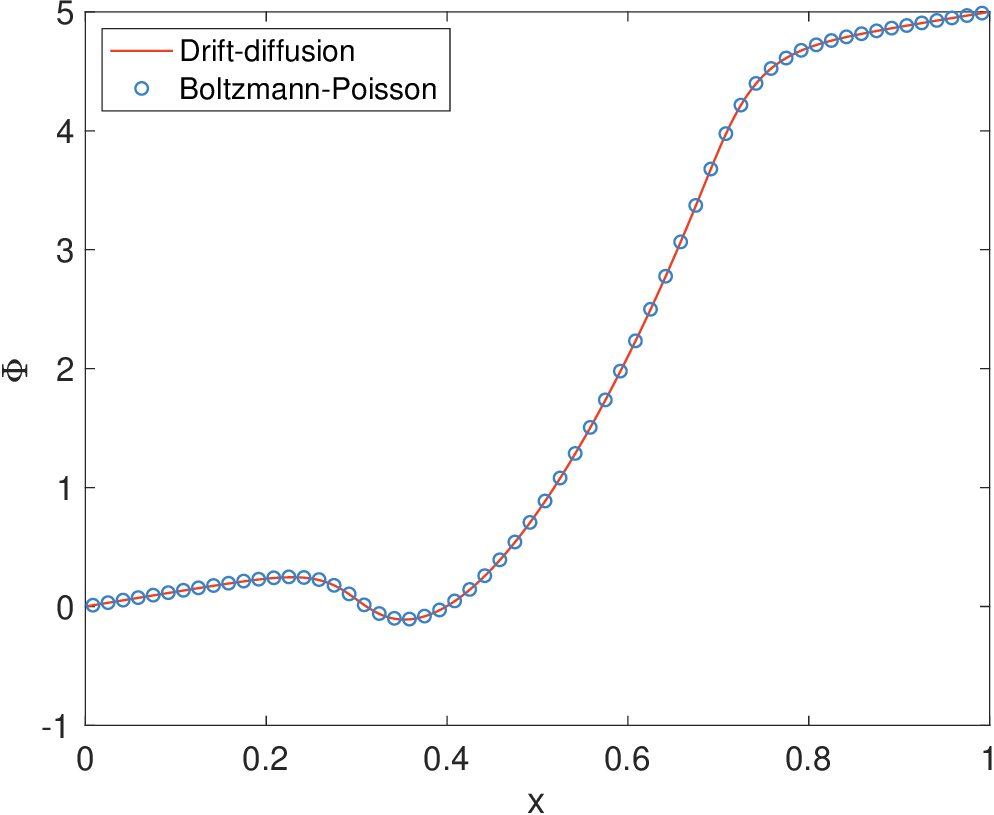}}
    \caption{Electric field $E$ (left) and electric potential $\Phi$ (right) for Boltzmann-Poisson system and the limiting drift-diffusion system in Example \ref{eg:BPSys}.}
    \label{fig:BPSysEpotential}
\end{figure}

\begin{exmp}[Mixed regime]\label{eg:mixing}
    In this example, we consider a complicated case where the Knudsen number $\varepsilon$ carries a mixed scaling and is spatially dependent as 
    \begin{equation*}
        \varepsilon(x)=10^{-3}+\frac{1}{2}[\tanh(1-11x)+\tanh(1+11x)].
    \end{equation*}
    This case contains both kinetic and diffusive regimes. The electric field is $E=0$. We assume the period boundary condition, with the initial condition given by 
    \begin{equation*}
        f(x,v,t=0)=\frac{\rho_0}{2}\left[
        \exp\left(-\frac{|v-u_0|^2}{T_0}\right)+\exp\left(-\frac{|v+u_0|^2}{T_0}\right)
        \right],
    \end{equation*}
    where $u_0=0.2$, $T_0=(5-2\cos(2\pi x))/20$, and $\rho_0=(2-\sin(2\pi x))/2$.
\end{exmp}

We set $\Delta x=0.02$ and $\Delta t=2\times 10^{-6}$ and simulate this example up to $t=0.1$. The reference solution is obtained by using the finite volume method with a fine mesh of $N_x=250$. We plot the density $\rho$ in Figure \ref{fig:mixing}(a). It can be observed that the numerical solution of the APDG scheme matches well with the reference solution, especially at the transition positions (in $x$), where $\varepsilon$ varies dramatically. This result demonstrates the excellent AP property of the  proposed method. 
 We also show in Figure \ref{fig:mixing}(b) the time evolution of the energy norm
\begin{equation}
    \mathcal{E}(t_n) = \sanshu r^{n} \sanshu^2 + \| \varepsilon \|_{L^2} \sanshu j^{n} \sanshu,
\end{equation}
which is decreasing with respect to time level $n$. This observation  is consistent with the stability analysis in Theorem \ref{Thm:main}.

\begin{figure}[!htbp]
    \centering
   \subfigure[Numerical solutions $\rho$] {\includegraphics[width=0.4\textwidth]{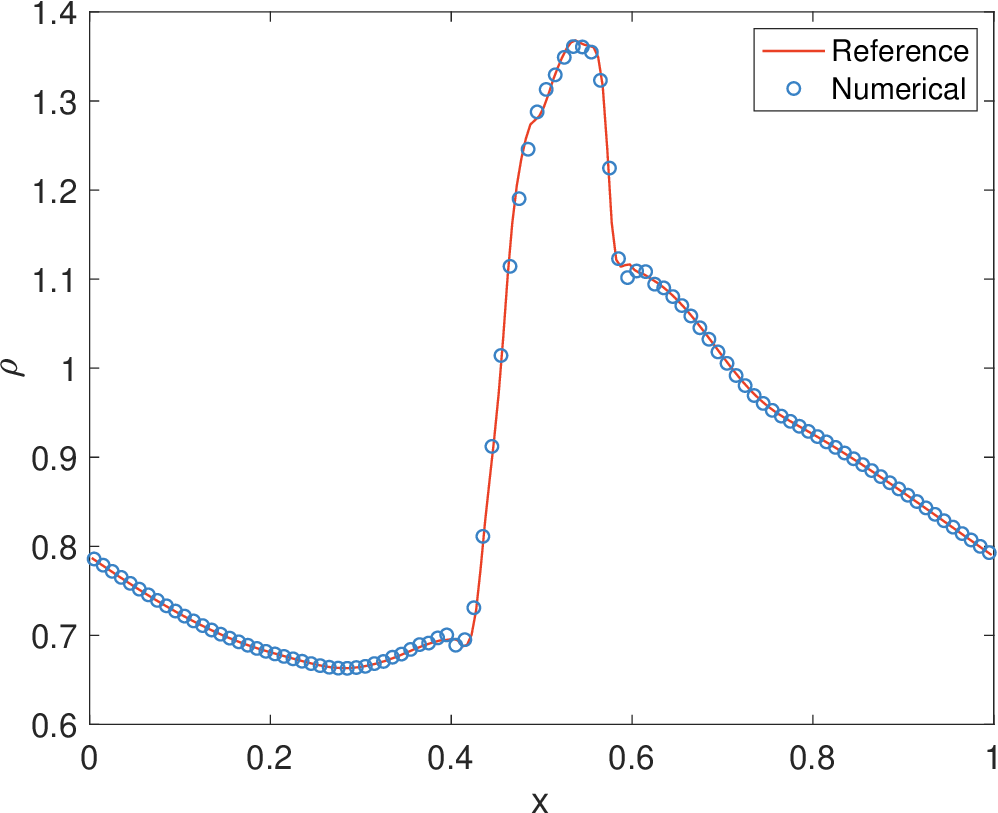}}\qquad\qquad
    \subfigure[Time evolution of energy norm] {\includegraphics[width=0.4\textwidth]{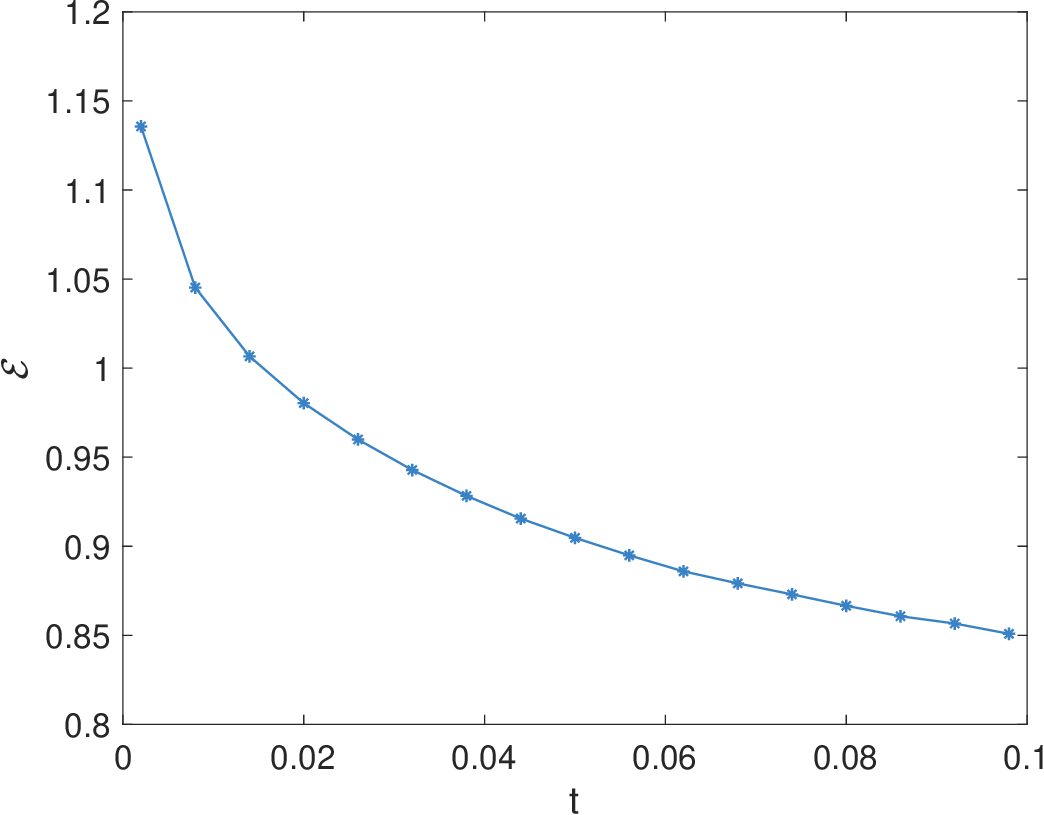}}
    \caption{ Numerical results for problems containing mixed regimes in Example \ref{eg:mixing}.}
    \label{fig:mixing}
\end{figure}


\section{Conclusion}
\label{sec:conclusion}
In this paper, we study the linear semiconductor Boltzmann equation under the diffusive scaling. An efficient and accurate numerical scheme that owns the property of asymptotic-preserving and positivity-preserving is proposed, with DG discretization in space and even-odd decomposition in time. We show the stability analysis in the even-odd decomposition framework and provide the CFL condition. Extensive numerical experiments have shown the accurate and robust performance of our proposed scheme. In the future work, we will develop higher-order IMEX Runge-Kutta methods and study other nonlinear kinetic equations.

\appendix
\section{Velocity discretization}
\label{sec:v_discrete}
In this section, we provide the velocity discretization using the Hermite polynomials, which has been widely used, see e.g. \cite{Klar, jin_discretization_2000}. 

Let $f(t,x,v) = \psi(t,x,v) M(v)$,  with $M(v)$ given in \eqref{eq:Maxwellian}
and
\begin{equation}\label{Psi} \psi(t,x,v) = \sum_{\ell=0}^{N_v} \psi_\ell(t,x) \tilde H_\ell(v), \end{equation}
where $\tilde H_\ell$ are the renormalized Hermite polynomials
defined as  $\tilde H_{-1}=0$, $\tilde H_0 = 1/(2\pi)^{1/4}$ and
$$ \tilde H_{\ell+1} = v \sqrt{\frac{1}{\ell+1}}\tilde H_\ell - \sqrt{\frac{\ell}{\ell+1}}\tilde H_{\ell-1} \quad \text{for } \ell \geq 0, $$
satisfying $\partial_v \tilde H_\ell = \sqrt{\ell}\, \tilde H_{\ell-1}$. 
The inverse Hermite expansion is given by 
\begin{equation}\label{I-Psi} 
\psi_i = \sum_{\ell=0}^{N_v}\psi(v_\ell)\, \tilde H_i(v_\ell)\, w_\ell, \quad i=0,\dots, N_v,
\end{equation}
where $\big\{v_\ell\big\}_{\ell=0}^{N_v}$ are the Gauss-Hermite quadrature points, and  $\big\{w_\ell\big\}_{\ell=0}^{N_v}$ are the associated quadrature weights. Here and after, we omit the arguments $x,t$ and adopt simpler notations such as $\psi_\ell$, $\psi(v)$, $f(v)$, etc.. 

The collision operator $Q$ can be computed by
$$ Q(f)(v) = M(v) \sum_{\ell=0}^{N_v} \sigma(v, v_\ell)\, \psi(v_\ell)\, w_\ell - \lambda(v) f(v), $$
with $\lambda(v) = \sum_{\ell=0}^{N_v} \sigma(v, v_\ell)\, w_\ell$. Therefore, we can compute the values of $Q(f)(v)$ on the Gauss-Hermite quadrature points as
\begin{equation}
    Q(f)(v_i) = M(v_i) \sum_{\ell=0}^{N_v} \sigma(v_i, v_\ell)\, \psi(v_\ell)\, w_\ell - \lambda(v_i) f(v_i).
\end{equation}

It follows from \eqref{Psi} and \eqref{I-Psi} that the derivative of $\psi$ with respect to $v$ by
\begin{equation*}
\begin{aligned}
\partial_v \psi & = \sum_{\ell=0}^{N_v} \psi_\ell\, \partial_v \tilde H_\ell(v) 
= \sum_{\ell=0}^{N_v} \psi_\ell \sqrt{\ell}\, \tilde H_{\ell-1}(v) \\
& = \sum_{\ell=0}^{N_v} \sum_{i=0}^{N_v} \psi(v_i) \, \tilde H_\ell(v_i) w_i \sqrt{\ell}\, \tilde H_{\ell-1}(v).
\end{aligned}
\end{equation*}
Thus, the derivative of $\psi$ with respect to $v$ at the Gauss-Hermite quadrature points can be further expressed as
\begin{equation}
    \partial_v \psi(v_m)= \sum_{\ell=0}^{N_v} \sum_{i=0}^{N_v} \psi(v_i) \, \tilde H_\ell(v_i) w_i \sqrt{\ell}\, \tilde H_{\ell-1}(v_m)
    =\sum_{i=0}^{N_v} \psi(v_i)\, C_{im},
\end{equation}
where $C_{im} = \sum_{\ell=0}^{N_v} \sqrt{\ell}\, \tilde H_\ell(v_i) \tilde H_{\ell-1}(v_m) w_i$
, which can be precomputed before the time iteration.

\bibliographystyle{siamplain}
\bibliography{APDG2}

\end{document}